\documentclass[10pt, twocolumn]{article}

\usepackage{amsmath}
\usepackage{amsfonts}
\usepackage{amsthm}
\usepackage[english]{babel}
\usepackage{graphicx}
\usepackage[all]{xy}
\usepackage{amssymb}

\newtheorem{theorem}{Theorem}
\newtheorem{lemma}{Lemma}
\newtheorem{definition}{Definition}
\newtheorem{corollary}{Corollary}
\newtheorem{proposition}{Proposition}
\newtheorem{remark}{Remark}

 \newtheorem{notation}[theorem]{Notation}

 \def\la{{\langle}}
 \def\ra{{\rangle}}

\newcommand{\mR}{\mathbb{R}}
\newcommand{\mC}{\mathbb{C}}
\newcommand{\mN}{\mathbb{N}}

\newcommand{\mS}{\mathbb{S}}

\newcommand{\mH}{\mathbb{H}}

\newcommand{\cM}{\mathcal{M}}

\newcommand{\cF}{\mathcal{F}}
\newcommand{\cP}{\mathcal{P}}
\newcommand{\cC}{\mathcal{C}}

\newcommand{\cS}{\mathcal{S}}

\newcommand{\cI}{\mathcal{I}}

\newcommand{\ux}{\underline{x}}

\newcommand{\uy}{\underline{y}}
\newcommand{\uu}{\underline{u}}

\newcommand{\ut}{\underline{t}}

\newcommand{\upx}{\partial_{\underline{x}}}

 \def\a{{\alpha}} 
 \def\b{{\beta}}

 \def\l{{\lambda}}

 \def\la{{\langle}}
 \def\ra{{\rangle}}

\begin{document}

\title{Convolution products for hypercomplex Fourier transforms}

\author{Roxana Bujack\footnote{bujack@informatik.uni-leipzig.de, Institut f\"ur Informatik, Universit\"at Leipzig, Johannisgasse 26, 04103 Leipzig, Germany.}, Hendrik De Bie\footnote{Hendrik.DeBie@UGent.be, Department of Mathematical Analysis, Faculty of Engineering and Architecture, 
Ghent University,
Galglaan 2, 9000 Ghent, Belgium.}, Nele De Schepper\footnote{Nele.DeSchepper@UGent.be, Department of Mathematical Analysis, Faculty of Engineering and Architecture, 
Ghent University,
Galglaan 2, 9000 Ghent, Belgium.}, Gerik Scheuermann\footnote{scheuermann@informatik.uni-leipzig.de, Institut f\"ur Informatik, Universit\"at Leipzig, Augustuplatz 10, 04109 Leipzig, Germany.}}

\date{}

\maketitle

\begin{abstract}
Hypercomplex Fourier transforms are increasingly used in signal processing for the analysis of higher-dimensional signals such as color images. A main stumbling block for further applications, in particular concerning filter design in the Fourier domain, is the lack of a proper convolution theorem. The present paper develops and studies two conceptually new ways to define convolution products for such transforms. As a by-product, convolution theorems are obtained that will enable the development and fast implementation of new filters for quaternionic signals and systems, as well as for their higher dimensional counterparts.
\end{abstract}


\section{Introduction}
\label{sec:1}

Recently, there has been an increased interest in applying hypercomplex Fourier transforms (FTs) in various aspects of signal processing where higher dimensional or vector signals are used, such as color image processing \cite{ES,MST}, flow visualization \cite{Ebl1,Ebl2}, and even spoken word recognition \cite{BTN}. The main idea behind these applications is the representation of a signal (say, a color image) as a pure quaternion or as an element of a suitable Clifford algebra (see Section \ref{prelim} for a precise definition). This representation is subsequently  analysed using a generalisation of the classical Fourier transform to a hypercomplex FT, which takes into account the multi-dimensional and multi-component nature of the signal under consideration. This stands in stark contrast to a component based classical analysis, sometimes also called marginal analysis. Successful further developments of the hypercomplex approach include the design of a color edge filter \cite{S_CE}, as well as other filters \cite{SE2}, construction of FFT methods to compute hypercomplex FTs \cite{PDC}, etc.

The main issue that hinders further development of applications (and in particular of filter design not based on ad hoc assumptions or ideas) is the lack of a suitable convolution theorem.  Indeed, in \cite{BU5} it was shown that hypercomplex FTs such as the quaternionic FT do not interact nicely with the classical convolution product, but rather lead to very complicated expressions in the Fourier domain. This means that, up to now, no filter design was possible in the Fourier domain, and hence that no fast implementations as multiplication operators have been obtained so far.

The main aim of the present paper is to tackle that problem. We will investigate, on theoretical grounds, the different possible convolution products that can be defined for a wide class of hypercomplex FTs. It turns out that two conceptual ways exist to achieve this, both having a left and right version. Before explaining these two definitions, let us recall that in the recent literature three different approaches to hypercomplex FTs have been considered. We can identify them as follows
\begin{itemize}
\item \textbf{A:} Eigenfunction approach 
\item \textbf{B:} Generalized roots of $-1$ approach
\item \textbf{C:} Characters of Spin group approach
\end{itemize}

The first approach is mainly studied in \cite{MR2190678,MR2283868,AIEP,DBNS,DBDSFr,DBDSC,H12,DBOSS,DBS,DBXu}, and aims at constructing new hypercomplex transforms by prescribing eigenvalues to a suitable basis of a Clifford-algebra valued $L_{2}$ space. The choice of suitable eigenvalues implies that there is a huge design freedom in this approach. The transforms of this class also have a deep connection with quantum mechanics and exhibit a very particular underlying algebraic structure, namely that of the Lie superalgebra $\mathfrak{osp}(1|2)$. For a recent review from this point of view, we refer the reader to \cite{DBR}.

The second approach is mainly advocated in \cite{BU,BU5} and
boils down to replacing the imaginary unit $i$ in the exponent of the ordinary Fourier transform by a generalized root of $-1$, belonging to a Clifford algebra (see e.g. \cite{HA,HHA} for a detailed study of such roots). It encompasses several of the hypercomplex FTs often used in applications, such as the quaternionic Fourier transform \cite{ES}, the Sommen-B\"ulow transform \cite{MR1875365,MR650399}, the Clifford Fourier transform (written without hyphen) introduced in \cite{Ebl2} and further extended in \cite{Maw1,Maw3}, and the cylindrical Fourier transform \cite{AIEP}. Again this approach exhibits a huge design freedom, as the set of roots of $-1$ is very big and as the roots can in principle be chosen independently for each application.

Finally, a third approach is given in \cite{B,B2}, and uses the notion of character (or group morphism) to generalize the ordinary Fourier transform to the setting of the group $Spin(3)$, resp. $Spin(5)$ for direct application in grey scale, resp. color image processing. Although the conceptual ideas of this third approach are very different from the second approach \textbf{B}, the resulting transforms can nevertheless be written as special cases of the transforms in \textbf{B}. For that reason we will only focus on the first two approaches in the sequel.

The classical FT and the classical convolution product will serve as a guide to define the generalized convolution products. Recall that the classical convolution product for two functions $f$ and $g$ is defined by
\[
f * g (x) : = \int_{\mR^m} f(y)\tau_y g(x) dy, 
\]
with $\tau_y g(x) := g(x-y)$. The classical FT, defined over $\mR^m$, is given by
\[
\cF(f)(y) := (2 \pi)^{-\frac{m}{2}} \int_{\mR^m} e^{-i \la x, y \ra} f(x) dx
\]
for $f \in L_1(\mR^m)$, where $ \la x,y\ra = \sum_{j=1}^m x_j y_j$ is the standard inner product. It interacts nicely with the convolution product. Namely, one has
\begin{equation}
\label{ordConvThm}
\cF(f * g ) = (2 \pi)^{m/2}  \cF(f) \cF(g).
\end{equation}
A first possibility to generalize the convolution product is hence obtained by taking the inverse FT of the right-hand side of (\ref{ordConvThm}) as a definition. This idea was first explored by Mustard for the fractional Fourier transform, see \cite{MR1668139}.

Another way to generalize the convolution product is obtained by introducing a generalization of the (geometric) translation operator $\tau_y$. This is also the strategy which is e.g. used for the Dunkl transform (see \cite{MR1973996,MR2274972}). Let us first illustrate the concept in some more detail for the ordinary FT. First we compute the FT of the translation over $z$ of a function $f$:
\begin{eqnarray*}
\cF (\tau_z f)(y) & = & (2 \pi)^{-\frac{m}{2}} \int_{\mR^m} e^{-i \la x, y \ra} f(x-z) dx\\
&=& e^{-i \la z, y \ra}  (2 \pi)^{-\frac{m}{2}} \int_{\mR^m} e^{-i \la x, y \ra} f(x) dx\\
&=& e^{-i \la z, y \ra} \cF(f)(y).
\end{eqnarray*}
This means that, formally, the ordinary translation is recovered via
\begin{equation}
\label{translationFT}
\tau_z f(u)  = \cF^{-1} \left( e^{-i \la z, y \ra} \cF(f)(y) \right).
\end{equation}
Here, the inverse FT acts on the $y$ variable and yields the $u$ variable.

Now let $K(x,y)$ be the integral kernel of a hypercomplex FT $\cF_K$ and $\widetilde{K(x,y)}$ be the kernel of its inverse $\cF_K^{-1}$. Upon replacing $e^{-i \la z, y \ra}$ by the kernel of the hypercomplex FT under consideration in formula (\ref{translationFT}), one obtains the definition of a generalized translation operator:
\begin{equation}
\label{gentrans}
\tau_y^K f(x) = \int_{\mR^m} \widetilde{K(\xi,x)} K(y,\xi) \cF_K(f)(\xi) d\xi.
\end{equation}
Note that this definition immediately reduces to geometric translation (i.e. formula (\ref{translationFT})) when $K$ is the kernel of the classical FT. 

We can hence summarize the four possible definitions for a generalized convolution product, related to a hypercomplex FT $\cF_K$:
\begin{itemize}

\item The first type of definition is inspired by convolution for the fractional Fourier transform. It is given by
\[
 f *_1 g(x):= \cF_K^{-1}\left( \cF_K(f) \cF_K(g)\right)(x)
\]
or
\[
 f *_2 g(x):= \cF_K^{-1}\left( \cF_K(g) \cF_K(f)\right)(x).
\]
These two versions in general do not coincide due to the non-commutativity of the Clifford algebra. However, one immediately observes that $ f *_1 g =  g *_2 f$. For symmetry reasons it is nevertheless convenient to introduce both versions.

This definition is clearly interesting as it forces a convolution theorem to hold for the hypercomplex FT under consideration. As we will see in the sequel, it is possible to derive explicit formulas for these new convolutions in terms of classical convolutions (see e.g. the subsequent Theorem \ref{mustardconvGFT}).

\item Using the generalized translation operator $\tau^{K}$ on the other hand, one can also define
\[
 f *_3 g(x):=   \int_{\mR^m} [\tau_y^K f(x)]  g(y)dy
\]
or
\[
 f *_4 g(x):= \int_{\mR^m} f(y) [\tau_y^K g(x) ]dy.
\]
Both definitions are again different due to the non-commutativity of the Clifford algebra. Moreover, in this case there is no immediate relation expressing $ f *_3 g$ in terms of $ f *_4 g$ or vice versa.  

This definition is interesting for a different reason: it allows to obtain inversion theorems for hypercomplex FTs in a way similar to Theorem 8.4 in \cite{DBXu}. 
\end{itemize}

For hypercomplex FTs, so far only definition $f *_3 g$ has been applied, in only one particular case, see \cite{DBXu}. In the present paper, we will perform a detailed study of the four different products mentioned above, for transforms belonging to approach \textbf{A} and \textbf{B}. In both cases, we will introduce specific notations to distinguish between the various definitions. An overview of these notations and the results obtained in the paper is given in Table \ref{tableconv}. Note that the methods developed in the paper can equally be extended to other hypercomplex FTs that have not been designed yet.

The paper is organized as follows.  After some preliminaries on Clifford algebras and analysis, we introduce in Section \ref{secCA} the most general hypercomplex FT that can be considered within the approach \textbf{A}. Its integral kernel takes the form of an infinite series in terms of Gegenbauer polynomials and Bessel functions. We discuss several examples and calculate the action of the associated transform on a basis of the relevant function space. We also construct the inverse of this general transform. In Section \ref{secConv} we define the four types of convolution for this general transform, and study their basic properties. Next, in Section \ref{secTrans}, we compute the action of the generalized translation operator on the special class of radial functions and prove that, for this special set of functions, the new translation coincides in many cases with ordinary translation. The proof of this result is technical, and the reader may wish to skip it during a first reading of the paper. In Section \ref{secB}, we turn our attention to the transforms belonging to approach \textbf{B}. We give the general definition and compute the eigenvalues and eigenfunctions (thus establishing a connection with the techniques used in approach \textbf{A}). Then we revisit the different convolution definitions in this context. Both in the case of Mustard convolution (Section \ref{GFTMustard}) and in the generalized translation operator approach (Section \ref{GFTtrans}), we can now obtain very explicit formulas for the new convolutions, which can moreover easily be implemented in software packages. In Section \ref{secQFT} we restate our results for the special case of the quaternionic Fourier transform (qFT). This is the most well-known hypercomplex FT used in engineering, and hence deserves a separate treatment. We end the paper with Table \ref{tableconv} where the results are summarized for the two approaches \textbf{A} and \textbf{B}.

\section{Preliminaries on Clifford algebras and analysis}
\label{prelim}

The Clifford algebra $\cC l_{0,m}$ over $\mR^{m}$ is the algebra generated by $e_{i}$, $i= 1, \ldots, m$, under the relations
\begin{align*}
&e_{i} e_{j} + e_{j} e_{i} = 0, \quad i \neq j\\
&e_{i}^{2} = -1.
\end{align*}
This algebra has dimension $2^{m}$ as a vector space over $\mR$. It can be decomposed as $\cC l_{0,m} = \oplus_{k=0}^{m} \cC l_{0,m}^{k}$
with $\cC l_{0,m}^{k}$ the space of $k$-vectors defined by
\[
\cC l_{0,m}^{k} := \mbox{span} \{ e_{i_{1}} \ldots e_{i_{k}}, i_{1} < \ldots < i_{k} \}.
\]
In the applied literature, Clifford algebras are usually called geometric algebras. For a detailed exposition from this point of view, including geometric interpretations and applications in computer vision, we refer the reader to \cite{Dorst}.

In the sequel, we will always consider functions $f$ taking values in $\cC l_{0,m}$, unless explicitly mentioned. Such functions can be decomposed as
\begin{displaymath} 
f = f_{0} + \sum_{i=1}^{m} e_{i}f_{i} + \sum_{i< j} e_{i} e_{j} f_{ij} + \ldots + e_{1} \ldots e_{m} f_{1 \ldots m} 
\end{displaymath}
with $f_{0}, f_{i}, f_{ij}, \ldots, f_{1 \ldots m}$ all real- or complex-valued functions on $\mathbb{R}^m$.

The Dirac operator is given by $\upx := \sum_{j=1}^{m} \partial_{x_{j}} e_{j}$ and the vector variable by $\ux := \sum_{j=1}^{m} x_{j} e_{j}$.
The square of the Dirac operator equals, up to a minus sign, the Laplace operator in $\mR^{m}$: $\upx^{2} = - \Delta$. For more information regarding Clifford analysis, we refer the reader to \cite{MR697564,MR1169463}.

Denote by $\cP$ the space of polynomials taking values in $\cC l_{0,m}$, i.e. 
$$
   \cP : = \mR[x_{1}, \ldots, x_{m}] \otimes \cC l_{0,m}.
$$ 
The space of homogeneous polynomials of degree $k$ is then denoted by $\cP_{k}$. 
The space $\cM_{k} : = \ker{\upx} \cap \cP_{k}$ is called the space of spherical monogenics 
of degree $k$.

Next we define the inner product and the wedge product of two vectors $\ux$ and $\uy$
\begin{align*}
\langle x, y \rangle &:= \sum_{j=1}^{m} x_{j} y_{j}\\
 \ux \wedge \uy &:= \sum_{j<k} e_{j}e_{k} (x_{j} y_{k} - x_{k}y_{j}).
\end{align*}

\begin{remark}
From now on, we will use the following convention for denoting variables, resp. vectors. When $x=(x_1, \ldots, x_m) \in \mR^m$ is used as a variable we will not underline it (in order not to overload notations). When it is used as a vector  $\ux := \sum_{j=1}^{m} x_{j} e_{j}$ involving Clifford multiplication, we use the underlined notation.
\end{remark}

We introduce two different bases for the space\\ $\cS(\mR^{m}) \otimes \cC l_{0,m}$, where 
$\cS(\mR^m)$ denotes the Schwartz space.  Define the functions $ \psi_{j,k,\ell}$ by
\begin{align}\label{basis}
\begin{split}
\psi_{2j,k,\ell} &:= L_{j}^{\frac{m}{2}+k-1}(|x|^{2}) M_{k}^{(\ell)} e^{-|x|^{2}/2},\\
\psi_{2j+1,k,\ell} &:= L_{j}^{\frac{m}{2}+k}(|x|^{2}) \ux M_{k}^{(\ell)} e^{-|x|^{2}/2},
\end{split}
\end{align}
where $j,k \in \mN$, $\{M_{k}^{(\ell)} \in \cM_{k}: \ell= 1, \ldots, \dim \cM_{k}\}$ is a basis for $\cM_{k}$,
and $L_{j}^{\alpha}$ are the Laguerre polynomials. The set $\{ \psi_{j,k,\ell}\}$ forms a basis of 
$\cS(\mR^{m}) \otimes \cC l_{0,m}$, see \cite{Spec}. This basis is called the Clifford-Hermite basis or the spherical basis. Alternatively, define the one-dimensional Hermite functions (see e.g. \cite{Sz}) by
\begin{eqnarray*}
\psi_k (x) &:= \left(x- \frac{d}{dx} \right)^k e^{-x^2/2}\\
&= H_k(x) e^{-x^2/2}
\end{eqnarray*}
for $k \in \mN$. Then the set $\{ \psi_{j_1,j_2, \ldots, j_m}\}$  with
\[
\psi_{j_1,j_2, \ldots, j_m} = \psi_{j_1}(x_1)\psi_{j_2}(x_2) \ldots \psi_{j_m} (x_m)
\]
and $j_1, \ldots, j_m \in \mN$ is also a basis of $\cS(\mR^{m}) \otimes \cC l_{0,m}$, called the tensor product or cartesian basis. Both bases interact nicely with the ordinary FT. One has
\begin{align}
\cF \left( \psi_{j,k,\ell}\right) & = (-i)^{j+k} \psi_{j,k,\ell}\\
\cF \left( \psi_{j_1,j_2, \ldots, j_m}\right) & = (-i)^{j_1 + \ldots + j_m} \psi_{j_1,j_2, \ldots, j_m}.
\label{eigCLFT}
\end{align}
It is interesting to compare this result with the subsequent Theorems \ref{eigenvalues} and \ref{eigenvalsGFT}.
For more information about these two bases, we refer the reader to \cite{CPLMS}.

\section{Convolution products in Approach A}
\label{secCA}

\subsection{Definition of the transforms, eigenfunctions, eigenvalues and inverse}
\label{gentran}

In this section we consider a general kernel of the following form
\begin{align}\label{genkernel}
\begin{split}
K(x,y) &= \left( A(w, \widetilde{z}) + (\ux \wedge \uy)  B(w, \widetilde{z}) \right) \\ 
& \quad \times e^{\frac{i}{2}( \cot \alpha) (|x|^2 + |y|^2)} 
\end{split}
\end{align}
with
\begin{align*}
A(w,\widetilde{z}) &= \sum_{k=0}^{+\infty} \alpha_{k} \ (\widetilde{z})^{-\l}J_{k+\l}(\widetilde{z}) C^{\l}_{k}(w)\\
B(w,\widetilde{z}) &=  \sum_{k=1}^{+\infty} \beta_{k} \ (\widetilde{z})^{-\l-1} J_{k+\l}(\widetilde{z}) C^{\l+1}_{k-1}(w)
\end{align*}
and $\alpha_{k}, \beta_{k} \in \mC$, $\widetilde{z} = (|x||y|)/ \sin{\alpha}$, $w=\langle x, y \rangle/(|x||y|)$, $\lambda=(m-2)/2$ and $\alpha \in \lbrack - \pi, \pi \rbrack$. Here, $J_{\nu}$ is the Bessel function and $C_k^{\l}$ the Gegenbauer polynomial. We exclude the cases where $\alpha=0$ or $\alpha = \pm \pi$.

The integral transform associated with this kernel is defined by
\begin{equation}\label{gentransform}
\cF_K\left( f \right) (y) = \rho_{\a,m} \int_{\mR^{m}} K(x,y) \ f(x) \ dx
\end{equation}
with 
\[
\rho_{\a,m}=(\pi(1-e^{-2i\alpha}))^{-m/2}
\]
and with $dx$ the standard Lebesgue measure on $\mathbb{R}^m$. The precise form of the kernel in formula (\ref{genkernel}) is inspired by the results obtained in \cite{DBDSFr}. In particular, it encompasses all previously studied kernels in the eigenfunction approach \textbf{A}.

\begin{remark}
In order to ensure good analytic behavior of the transform (\ref{gentransform}), we additionally demand that both $A(w,\widetilde{z})$ and $B(w,\widetilde{z})$ satisfy polynomial bounds of the following type
\begin{align*}
|A(w,\widetilde{z})| \leq c (1+|x|)^{j}(1+|y|)^{j}, \\
|B(w,\widetilde{z})| \leq c (1+|x|)^{j}(1+|y|)^{j}, 
\end{align*}
with $j \in \mN$ and  $c$ a constant.

This is the case for all subsequent examples, and allows us to prove that the transform (\ref{gentransform}) yields a continuous map on $\cS(\mR^{m}) \otimes \cC l_{0,m}$, by adaptation of the proof in \cite{DBDSFr}.
\end{remark}

\subsubsection{Some examples}

In the special case where $\alpha=\pi/2$, the kernel takes the form
\begin{equation}\label{specialcase}
K(x,y) = A(w, z) + (\ux \wedge \uy) \  B(w, z) 
\end{equation}
with
\begin{align*}
A(w,z) &= \sum_{k=0}^{+\infty} \alpha_{k} \ z^{-\l} J_{k+\l}(z) C^{\l}_{k}(w)\\
B(w,z) &=  \sum_{k=1}^{+\infty} \beta_{k} \ z^{-\l-1} J_{k+\l}(z) C^{\l+1}_{k-1}(w)
\end{align*} 
and $z = |x||y|$, $w=\langle x, y \rangle/(|x||y|)$, $\lambda=(m-2)/2$. The corresponding integral transform is given by
\begin{align*}
\cF_K\left( f \right) (y) &= \rho_{\frac{\pi}{2},m} \int_{\mR^{m}} K(x,y) \ f(x) \ dx
\end{align*} 
where $\rho_{\frac{\pi}{2},m} = (2 \pi)^{-m/2}$.

Note that the kernel of the classical Fourier transform $\cF$, which can equally be expressed as the operator exponential $e^{ \frac{i \pi m}{4}}e^{ \frac{i \pi}{4}(\Delta - |x|^{2} )}$, takes the form (\ref{specialcase}) with
\begin{align*}
&\alpha_k = 2^{\l} \Gamma(\l) (k+\l) (-i)^k\\
&\beta_k = 0.
\end{align*}
Also the Clifford-Fourier transform (see \cite{MR2190678,MR2283868,AIEP,DBXu}), a generalization of the classical Fourier transform in the framework of Clifford analysis, takes this form. It is defined by the following exponential operator 
\[
\cF_{\pm} := e^{ \frac{i \pi m}{4}} e^{ \frac{i \pi}{4}(\Delta - |x|^{2} \mp 2 \Gamma)},
\] 
with 
\[
\Gamma := - \sum_{j<k} e_{j}e_{k} (x_{j} \partial_{x_{k}} - x_{k}\partial_{x_{j}}).
\]
In case of the Clifford-Fourier transform $\cF_{-}$, the coefficients $\alpha_k$ and $\beta_k$ in the kernel (\ref{specialcase}) take the form:
\begin{align*}
\alpha_k &= 2^{\l - 1} \Gamma(\l + 1) (i^{2\l+2}+(-1)^k)\\
&\quad  - 2^{\l - 1} \Gamma(\l) \ (k+\l) (i^{2\l+2}-(-1)^k)\\
\beta_k &= - 2^{\l} \Gamma(\l+1) (i^{2\l+2}+(-1)^k).
\end{align*}
For the transform $\cF_{+}$, similar expressions hold. Moreover, in \cite{DBNS}, an entire class of kernels of the form (\ref{specialcase}), for particular values of the coefficients $\alpha_k$ and $\beta_k$, was determined. They yield new integral transforms that have the same calculus properties (i.e. interaction with the Dirac operator) as the original Clifford-Fourier transform, but with different spectrum.

Also for general $\alpha$, concrete examples have been studied. The fractional Fourier transform (see \cite{OZA}) is a generalization of the classical Fourier transform. It is usually defined using the operator expression
\begin{displaymath}
\cF_{\alpha} = e^{ \frac{i \alpha m}{2}}e^{ \frac{i \alpha}{2}(\Delta - |x|^{2} )}, \qquad \alpha \in \lbrack - \pi ,\pi \rbrack.
\end{displaymath}
Recently, a fractional version of the Clifford-Fourier transform was introduced (see \cite{DBDSFr}). It is defined by the following exponential operator
\begin{displaymath}
\cF_{\alpha, \beta} = e^{ \frac{i \alpha m}{2}} e^{i\beta \Gamma} e^{ \frac{i \alpha}{2}(\Delta - |x|^{2} )}, \qquad \alpha, \beta \in \lbrack - \pi ,\pi \rbrack.
\end{displaymath}
The integral kernel of this transform takes the form (\ref{genkernel}) with
\begin{align*}
\alpha_k &=  2^{\l -1} \Gamma(\l) (k+\l) i^{-k} ( e^{i \beta (k+2\l)} + e^{-i\beta k })\\
&\quad-2^{\l-1} \Gamma(\l+1) i^{-k} (e^{i\beta (k+2\l)} - e^{-i\beta k})\\
\beta_k &= \frac{2^{\l} \Gamma(\l+1)}{\sin{\alpha}} i^{-k} (e^{i \beta (k+2\l)} - e^{-i\beta k}).
\end{align*}

\subsubsection{Eigenvalues}

Now we calculate the action of the transform (\ref{gentransform}) on the basis (\ref{basis}) of $\cS(\mR^{m}) \otimes \cC l_{0,m}$. We start with the following auxiliary result, expressing the radial behavior of the integral transform.
\begin{proposition}
\label{radbehavior1}
Let $M_{k} \in \cM_{k}$ be a spherical monogenic of degree $k$. Let $f(x)= f_0(|x|)$ be a real-valued radial function in $\cS(\mR^m)$. Further, put $\underline{\xi}= \ux/|x|$ and $\underline{\eta} = \uy/|y|$. Then one has, putting $\beta_0 = 0$,
\begin{eqnarray*} 
&&\cF_K \left( f(r)M_{k} \right) (y)\\
 &=& c_m \left( \frac{\l}{\l+k} \alpha_{k} - \sin{\alpha} \ \frac{k}{2(k+ \l)} \beta_k \right) e^{\frac{i}{2} (\cot \alpha) |y|^2} M_{k}(\eta)\\
&&\times \int_{0}^{+\infty} r^{m+k-1}f_0(r)  \ (\widetilde{z})^{-\l} J_{k + \l}(\widetilde{z}) \ e^{\frac{i}{2} (\cot \alpha) r^2} dr
\end{eqnarray*}
and
\begin{eqnarray*}
&&\cF_K \left( f(r) \ux M_{k} \right) (y) \\
&=& c_m \left( \frac{\l}{\l+k+1} \alpha_{k+1} + \sin{\alpha} \frac{k+1+2\l}{2(k+1+ \l)} \beta_{k+1} \right) \\
&&\times e^{\frac{i}{2} (\cot \alpha) |y|^2} \underline{\eta} \  M_{k}(\eta)  \\
&&\times  \int_{0}^{+\infty} r^{m+k}f_0(r)   \ (\widetilde{z})^{-\l} J_{k +1+ \l}(\widetilde{z}) \ e^{\frac{i}{2} (\cot \alpha) r^2} dr
\end{eqnarray*}
with $\widetilde{z}= (r |y|)/\sin{\alpha}$, $\l = (m-2)/2$ and 
\[
c_m =  \frac{2}{\Gamma \left( \frac{m}{2} \right) \ (1-e^{-2i\alpha})^{m/2}}.
\]
\end{proposition}
\begin{proof}
The proof goes along similar lines as the proof of Theorem 6.4 in \cite{DBXu}.
 \end{proof}

We then have the following theorem.
\begin{theorem}
\label{eigenvalues}
One has, putting $\beta_{0}=0$,
\begin{align*}
&\cF_K ( \psi_{2j,k,\ell})\\  & = \frac{2^{-\l}}{\Gamma(\l+1)} \ \left(\frac{\l}{\l+k} \alpha_{k} - \sin{\alpha} \ \frac{k}{2(\l+k)} \beta_{k} \right)\\
&\quad \times i^k e^{-i \alpha (k+2j)} \ \psi_{2j,k,\ell}
\end{align*}
and
\begin{align*}
&\cF_K ( \psi_{2j+1,k,\ell}) \\ & = \frac{2^{-\l}}{\Gamma(\l+1)}  \left(\frac{\l}{\l+k+1} \alpha_{k+1} + \sin{\alpha}  \frac{k+1+2\l}{2(\l+k+1)} \beta_{k+1} \right)\\ 
&\quad  \times  i^{k+1} e^{-i\alpha(k+2j+1)} \ \psi_{2j+1,k,\ell}.
\end{align*}
\end{theorem}

\begin{proof}
This follows from the explicit expression (\ref{basis}) of the basis and the identity (see \cite[p. 847, formula 7.421, number 4 with $\alpha=1$]{Grad}):
\begin{align*}
&\int_{0}^{+\infty} x^{\nu+1} e^{-\b x^2} L_{n}^{\nu}(x^{2}) J_{\nu}(xy) dx\\
& = 2^{-\nu-1} \b^{-\nu-n-1} (\b-1)^n y^{\nu} e^{-\frac{y^2}{4 \b}} L_n^{\nu} \left( \frac{y^2}{4 \b (1-\b)} \right). 
\end{align*}
 \end{proof}

\begin{remark}
Theorem \ref{eigenvalues} is very important; it allows us to design a hypercomplex Fourier transform $\cF_K$ by prescribing the eigenvalues on the basis $\{\psi_{j,k,\ell}\}$ via 
\begin{align*}
\cF_K(\psi_{2j,k,\ell}) &= \lambda_{k}e^{-i \alpha 2j} \psi_{2j,k,\ell}\\
\cF_K(\psi_{2j+1,k,\ell}) &= \mu_{k}e^{-i \alpha (2j+1)} \psi_{2j+1,k,\ell}
\end{align*}
for any set of numbers $\lambda_{k}, \mu_{k} \in \mC$.
Indeed, it suffices to solve the system of equations
\begin{align*}
\lambda_{k}&= \frac{2^{-\l}}{\Gamma(\l+1)}  \left(\frac{\l}{\l+k} \alpha_{k} - \sin{\alpha} \ \frac{k}{2(\l+k)} \beta_{k} \right) i^k e^{-i \alpha k}\\
\mu_{k}&=   \left(\frac{\l}{\l+k+1} \alpha_{k+1} + \sin{\alpha} \ \frac{k+1+2\l}{2(\l+k+1)} \beta_{k+1} \right)\\
& \quad \times\frac{2^{-\l}}{\Gamma(\l+1)}  i^{k+1} e^{-i \alpha k} 
\end{align*}
to determine the integral kernel $K(x,y)$ in terms of the coefficients $\alpha_k$ and $\beta_k$.
\end{remark}


\subsubsection{Inverse transform}
\label{SectionInverse}
In order to construct the inverse of the general transform $\cF_K$ on the basis $\lbrace \psi_{j,k,\ell} \rbrace$ we consider the following integral transform:
\begin{align*}
\cF_K^{\ast} \left( f \right) (y) &=\rho_{-\a,m} \int_{\mR^{m}} K^{\ast}(x, y) \ f(x) \ dx,
\end{align*}
where the kernel is given by
\begin{align*}
K^{\ast}(x, y) &= \left( A^{\ast}(w, \widetilde{z}) + (\ux \wedge \uy)   B^{\ast}(w, \widetilde{z}) \right) \\
& \quad \times e^{-\frac{i}{2}( \cot \alpha) (|x|^2 + |y|^2)} 
\end{align*}
with
\begin{align*}
A^{\ast}(w,\widetilde{z}) &= \sum_{k=0}^{+\infty} (-1)^k \gamma_{k} \ (\widetilde{z})^{-\l}J_{k+\l}(\widetilde{z}) C^{\l}_{k}(w)\\
B^{\ast}(w,\widetilde{z}) &=  \sum_{k=1}^{+\infty} (-1)^{k+1} \ \delta_{k} \ (\widetilde{z})^{-\l-1} J_{k+\l}(\widetilde{z}) C^{\l+1}_{k-1}(w)
\end{align*}
and $\gamma_{k}, \delta_{k} \in \mC$, $\widetilde{z}= (|x|  |y|)/\sin{\alpha}$, $w=\langle x, y \rangle/(|x||y|)$, $\lambda=(m-2)/2$.

Similarly as for the transform $\cF_K$, we can consecutively calculate the radial behavior of the transform $\cF_K^{\ast}$ and determine its action on the basis $\lbrace \psi_{j,k,\ell} \rbrace$. This yields the following result.
\begin{theorem}
\label{eigenvalues2}
One has, putting $\delta_{0}=0$,
\begin{align*}
&\cF_K^{\ast} ( \psi_{2j,k,\ell})\\
 & = \frac{2^{-\l}}{\Gamma(\l+1)}  \left(\frac{\l}{\l+k} \gamma_{k} + \sin{\alpha} \ \frac{k}{2(\l+k)} \delta_{k} \right)\\
 & \quad \times i^k e^{i \alpha (k+2j)} \ \psi_{2j,k,\ell}
 \end{align*}
and
\begin{align*}
&\cF_K^{\ast} ( \psi_{2j+1,k,\ell})\\
  & = \frac{2^{-\l}}{\Gamma(\l+1)}  \left(\frac{\l}{\l+k+1} \gamma_{k+1} - \sin{\alpha}  \frac{k+1+2\l}{2(\l+k+1)} \delta_{k+1} \right) \\ 
&\quad  \times i^{k+1}e^{i\alpha(k+2j+1)} \ \psi_{2j+1,k,\ell}.
\end{align*}
\end{theorem}

\begin{proof}
Similar to the proof of Theorem \ref{eigenvalues}.
 \end{proof}

Combining Theorem \ref{eigenvalues} and \ref{eigenvalues2}, we are now able to construct the inverse of the general transform $\cF_K$ on the basis $\lbrace \psi_{j,k,\ell} \rbrace$.
\begin{theorem}\label{inverse}
The inverse of $\cF_K$ on the basis $\lbrace \psi_{j,k,\ell} \rbrace$ is given by
\begin{align*}
\cF_K^{-1} \left( f \right) (y) &=\rho_{-\a,m} \int_{\mR^{m}} \widetilde{K(x,y)} \ f(x) \ dx,
\end{align*}
with
\begin{align*}
\widetilde{K(x,y)} &= \left( \widetilde{A(w, \widetilde{z})} + (\ux \wedge \uy)   \widetilde{B(w, \widetilde{z})} \right) \\
& \quad \times e^{-\frac{i}{2}( \cot \alpha) (|x|^2 + |y|^2)} 
\end{align*}
where
\begin{align*}
\widetilde{A(w,\widetilde{z})} &= \sum_{k=0}^{+\infty} \frac{1}{N_k ^{\lambda}} ( \alpha_k+\beta_k \ \sin{\alpha}) \ (\widetilde{z})^{-\l}J_{k+\l}(\widetilde{z}) C^{\l}_{k}(w)\\
\widetilde{B(w,\widetilde{z})} &= -   \sum_{k=1}^{+\infty} \frac{1}{N_k ^{\lambda}} \beta_k \ (\widetilde{z})^{-\l-1} J_{k+\l}(\widetilde{z}) C^{\l+1}_{k-1}(w),
\end{align*}
and
\begin{align*}
N_k ^{\lambda}&= \frac{1}{2^{2\l} (\Gamma (\l +1))^2} \left(\frac{\l}{\l+k} \alpha_{k} - \sin{\alpha}  \frac{k}{2(\l+k)} \beta_{k} \right)\\
& \quad \times \left(\frac{\l}{\l+k} \alpha_{k} + \sin{\alpha} \frac{k+2\l}{2(\l+k)} \beta_{k} \right).
\end{align*}
\end{theorem}
\begin{proof}
Put $\widetilde{K(x,y)} = \left( \widetilde{A(w,\widetilde{z})} + (\ux \wedge \uy)   \widetilde{B(w,\widetilde{z})} \right)$\\ $\times e^{-\frac{i}{2}( \cot \alpha) (|x|^2 + |y|^2)} $ where
\begin{align*}
\widetilde{A(w,\widetilde{z})} &= \sum_{k=0}^{+\infty} (-1)^k \gamma_{k} \ (\widetilde{z})^{-\l} J_{k+\l}(\widetilde{z}) C^{\l}_{k}(w)\\
\widetilde{B(w,\widetilde{z})} &=  \sum_{k=1}^{+\infty} (-1)^{k+1} \delta_{k} \ (\widetilde{z})^{-\l-1} J_{k+\l}(\widetilde{z}) C^{\l+1}_{k-1}(w)
\end{align*}
and with $\gamma_{k}, \delta_{k} \in \mC$. We need to have that
\[
\cF_K^{-1} \bigl( \cF_K \left( f \right) \bigr) = \cF_K \bigl( \cF_K^{-1} \left( f \right) \bigr) = f.
\]
Using Theorem \ref{eigenvalues} and  \ref{eigenvalues2}, this condition is equivalent with the system of equations ($k=0,1, \ldots$)
\begin{align*}
&\left(\frac{\l}{\l+k} \alpha_{k} - \sin{\alpha} \frac{k}{2(\l+k)} \beta_{k} \right)\\
& \times \left(\frac{\l}{\l+k} \gamma_{k} + \sin{\alpha} \frac{k}{2(\l+k)} \delta_{k} \right) =(-1)^k (\Gamma(\l+1))^2 2^{2 \l}
\end{align*}
and
\begin{align*}
&\left(\frac{\l}{\l+k} \alpha_{k} + \sin{\alpha} \frac{k+2\l}{2(\l+k)} \beta_{k} \right)\\
& \times \left(\frac{\l}{\l+k} \gamma_{k} - \sin{\alpha} \frac{k+2\l}{2(\l+k)} \delta_{k} \right)= (-1)^k (\Gamma(\l+1))^2 2^{2 \l}.
\end{align*}
Solving these two equations for $\gamma_k$ and $\delta_k$ then yields the statement of the theorem.
 \end{proof}


\subsection{Four types of generalized convolution}
\label{secConv}

\subsubsection{Definitions based on an idea of Mustard}
The definitions of the first two types of convolutions are based on the observation that in the classical case the following interaction between the convolution and the Fourier transform holds:
\begin{displaymath}
\cF ( f \ast g ) = (2 \pi)^{m/2} \cF\left( f \right) \ \cF\left( g \right).
\end{displaymath}
In our case, this leads to the following:
\begin{definition}\label{conv2} 
For $f,g \in \cS(\mR^{m}) \otimes \cC l_{0,m}$, the generalized convolution $f \ast_{C,L} g$ is defined for $x \in \mathbb{R}^m$ by
\begin{displaymath}
\left( f \ast_{C,L} g \right)(x):= \rho_{\a,m}^{-1} \cF_K^{-1} \left( \cF_K\left( f \right) \ \cF_K \left( g \right) \right) (x).
\end{displaymath}
Similarly, the generalized convolution $f \ast_{C,R} g$ takes the form 
\begin{displaymath}
\left( f \ast_{C,R} g \right)(x):= \rho_{\a,m}^{-1}  \cF_K^{-1} \left( \cF_K\left( g \right) \ \cF_K \left( f \right) \right) (x).
\end{displaymath}
\end{definition}
Note that, as already mentioned in the introduction, $f \ast_{C,R} g = g \ast_{C,L} f$. 

Taking into account the integral expression for $\cF_K$ and its inverse $\cF_K^{-1}$, see formula (\ref{gentransform}) and Theorem \ref{inverse}, we obtain the following explicit formulas.
\begin{proposition}\label{explicitconv2}
The generalized convolutions $f \ast_{C,L} g$ and $f \ast_{C,R} g$ take the following explicit form:
\begin{align*}
\left(f \ast_{C,L} g \right)(x)
&=c_{\alpha, m} \int_{\mathbb{R}^m} \int_{\mathbb{R}^m} \int_{\mathbb{R}^m} \widetilde{K(u,x)}\\
& \quad \times K(t,u) f(t) K(y,u) g(y)  dt dy du
\end{align*}
and
\begin{align*}
\left(f \ast_{C,R} g \right)(x)
&=c_{\alpha, m} \int_{\mathbb{R}^m} \int_{\mathbb{R}^m} \int_{\mathbb{R}^m} \widetilde{K(u,x)}\\
& \quad \times K(y,u) g(y) K(t,u) f(t) dt du dy,
\end{align*}
with
\begin{align*}
c_{\alpha, m}&= \rho_{\a,m} \rho_{-\a,m}\\
&=\left( \pi (1-e^{2i\alpha}) \right)^{-m/2} \left( \pi (1-e^{-2i\alpha}) \right)^{-m/2}\\
& = (2 \pi)^{-m} |\sin{\alpha}|^{-m}.
\end{align*}
\end{proposition}

\subsubsection{Definitions using the generalized translation operator}

We first define a generalized translation operator related to the integral transform $\cF_K$ defined in Section \ref{gentran}. 
\begin{definition}
\label{defTrans}
Let $f \in \cS(\mR^{m}) \otimes \cC l_{0,m}$. For $y \in \mathbb{R}^m$ the generalized translation operator $f \longmapsto \tau^K_{y} f$ is defined by
\begin{displaymath}
\cF_K ( \tau^K_{y} f )(x)= K(y,x) \ \cF_K\left( f \right) (x), \qquad x \in \mathbb{R}^m.
\end{displaymath}
\end{definition}
It can be expressed, by the inverse of $\cF_K$ (see Section \ref{SectionInverse}), as an integral operator
\begin{equation}\label{exprtrans}
\tau^K_{y} f(x) = \rho_{-\a,m}\int_{\mathbb{R}^m} \widetilde{K(\xi,x)} \ K(y,\xi) \ \cF_K \left( f \right) (\xi) \ d\xi.
\end{equation}

Using this generalized translation, we can again define two types of convolution for functions with values in the Clifford algebra.
\begin{definition}\label{conv1} 
For $f,g \in \cS(\mR^{m}) \otimes \cC l_{0,m}$, the generalized convolution $f \ast_L g$, resp. $f \ast_R g$, is defined for $x \in \mathbb{R}^m$ by 
\begin{displaymath}
\left( f \ast_L g \right)(x):=  \int_{\mathbb{R}^m} [ \tau^K_{y}f(x) ] \ g(y) \ dy\end{displaymath}
resp.
\begin{displaymath}
\left( f \ast_R g \right)(x):= \int_{\mathbb{R}^m} f(y) [ \tau^K_{y}g(x) ] dy.
\end{displaymath}
\end{definition}

Using the integral expression for the generalized translation operator (see (\ref{exprtrans})):
\begin{align*}
\tau^K_{y}f(x) &= c_{\alpha, m}  \int_{\mathbb{R}^m} \int_{\mathbb{R}^m} \widetilde{K(u,x)} K(y,u) K(t,u) f(t) dt du,
\end{align*}
we obtain the explicit formulas for the generalized convolutions introduced above.
\begin{proposition}\label{explicitconv1}
The convolutions $f \ast_L g$ and $f \ast_R g$ take the following explicit form:
\begin{align*}
\left(f \ast_L g \right)(x)&= c_{\alpha, m}  \int_{\mathbb{R}^m} \int_{\mathbb{R}^m} \int_{\mathbb{R}^m} \widetilde{K(u,x)}\\
& \quad \times K(y,u) K(t,u) f(t) g(y) dt du dy
\end{align*}
and
\begin{align*}
\left(f \ast_R g \right)(x)&=c_{\alpha, m} \int_{\mathbb{R}^m} \int_{\mathbb{R}^m} \int_{\mathbb{R}^m} f(t) \widetilde{K(u,x)} \\
& \quad \times K(t,u) K(y,u) g(y) dy du dt.
\end{align*}
\end{proposition}

\subsubsection{Connection between the four types of convolution and further properties}
\label{connectionconvolutions}

For scalar functions the four types of convolution defined in the previous subsections are strongly related.
\begin{proposition}\label{connection}
Let $f,g \in \cS(\mR^{m})$ be scalar functions. Then one has
\begin{itemize}
\item[(i)] $\left( f \ast_{C,L} g \right)(x) = \left( f \ast_{R} g \right)(x)$
\item[(ii)] $\left( f \ast_{C,R} g \right)(x) = \left( f \ast_{L} g \right)(x)$
\item[(iii)] $\left( f \ast_{L} g \right)(x) = \left( g \ast_{R} f \right)(x)$
\item[(iv)] $\left( f \ast_{C,L} g \right)(x) = \left( g \ast_{C,R} f \right)(x).$
\end{itemize}
\end{proposition}

\begin{proof}
These relations follow immediately from the explicit formulas of the convolutions (see Proposition  \ref{explicitconv2} and Proposition \ref{explicitconv1}).
 \end{proof}

The following proposition quantifies the difference between the left and right versions of the generalized convolutions.
\begin{proposition}
Let $f,g \in \cS(\mR^{m})$ be scalar functions. Then one has
\begin{align*}
&\left( f \ast_{L} g \right)(x) - \left( f \ast_{R} g \right)(x)\\
 &= \left( f \ast_{C,R} g \right)(x) - \left( f \ast_{C,L} g \right)(x)\\
&=c_{\alpha, m}
\int_{\mathbb{R}^m} \int_{\mathbb{R}^m} \int_{\mathbb{R}^m} \widetilde{K(u,x)} \left\lbrack K(y,u) , K(t,u) \right\rbrack\\
&\quad \times f(t) g(y) dt du dy,
\end{align*}
where the commutator of $K(y,u)$ and $K(t,u)$ takes the form
\begin{align*}
\left\lbrack K(y,u) , K(t,u) \right\rbrack &= B(w_1,\widetilde{z_1}) B(w_2,\widetilde{z_2})\\
& \quad \times e^{\frac{i}{2} \cot{\alpha}(|y|^2+|t|^2+2|u|^2)}\\
& \quad \times \left( (\uy \wedge \uu)(\ut \wedge \uu) - (\ut \wedge \uu)(\uy \wedge \uu) \right)
\end{align*}
with $\widetilde{z_1} = (|y| |u|)/\sin{\alpha}$, $\widetilde{z_2} = (|t| |u|)/\sin{\alpha}$, $w_1= \la \xi_1,\eta_1 \ra$, $w_2=\la \xi_2,\eta_1 \ra$ where $y = |y| \xi_1$, $u = |u| \eta_1$and $t = |t| \xi_2$.
\end{proposition}

From the definition of the convolutions $f \ast_{C,L} g$ and $f \ast_{C,R} g$ we immediately obtain the following result.
\begin{theorem}
For $f,g \in \cS(\mR^{m}) \otimes \cC l_{0,m}$, one has
\begin{displaymath}
\cF_K ( f \ast_{C,L} g ) = \rho_{\a,m}^{-1}   \cF_K \left( f \right)  \cF_K \left( g \right)
\end{displaymath}
and
\begin{displaymath}
\cF_K ( f \ast_{C,R} g ) = \rho_{\a,m}^{-1} \cF_K \left( g \right)  \cF_K \left( f \right).
\end{displaymath}
\end{theorem}
Taking into account Proposition \ref{connection} we obtain in case of scalar functions a similar result for the convolutions $f \ast_{L} g$ and $f \ast_{R} g$.
\begin{proposition} 
For $f,g \in \cS(\mR^{m})$ scalar functions, one has
\begin{displaymath}
\cF_K ( f \ast_{L} g ) =\rho_{\a,m}^{-1}   \cF_K \left( g \right)  \cF_K \left( f \right)
\end{displaymath}
and
\begin{displaymath}
\cF_K ( f \ast_{R} g ) = \rho_{\a,m}^{-1} \cF_K \left( f \right)  \cF_K \left( g \right).
\end{displaymath}
\end{proposition}


\subsection{Generalized translation of radial functions}
\label{secTrans}

The generalized translation defined in formula (\ref{exprtrans}) no longer equals geometric translation. However, in the special case when we compute the translation of a radial function and when $\a = \pi/2$, we still find that generalized translation equals geometric translation. That is the main result we will obtain in this section. As a consequence, we will be able to give more results on the new convolution products when one of the functions is radial.

The key ingredient we need is a compact formula for the integral over the unit sphere
\begin{displaymath}
\int_{\mS^{m-1}} \widetilde{K(r \eta, x)} \ K(y, r \eta) \ d \omega (\eta)
\end{displaymath}
where $\int_{\mS^{m-1}} d \omega (\eta) = 1$. This is derived using the series representation of the kernel function and some lemmas which are proven in \cite{DBXu}. We state them again for the convenience of the reader.

\begin{lemma}\label{lemma1}
For $k,\ell \in \mathbb{N}$, $\l =(m-2)/2$ and $x' , y' \in \mS^{m-1}$, one has
\begin{align*}
&\int_{\mS^{m-1}} (\uy' \wedge \underline{\eta}) \ C_k^{\l + 1}( \langle \eta,y' \rangle) \ C_{\ell}^{\l}( \langle \eta,x' \rangle) \ d\omega(\eta)\\
& = - \frac{\l}{k+1+\l} \delta_{\ell-1,k} (\ux' \wedge \uy') \ C_k^{\l+1}( \langle x' , y' \rangle).
\end{align*}
\end{lemma}
\begin{lemma}\label{lemma2}
For $k,\ell \in \mathbb{N}$, $\l =(m-2)/2$ and $x' , y' \in \mS^{m-1}$, one has
\begin{align*}
& \int_{\mS^{m-1}} (\underline{\eta} \wedge \ux') \ C_k^{\l + 1}( \langle\eta, x' \rangle) \ (\uy' \wedge \underline{\eta}) C_{\ell}^{\l+1}( \langle \eta,y' \rangle) \ d\omega(\eta)\\
 &= \delta_{k,\ell} \frac{(k+1)(k+1+2\l)}{4 \l (k+\l+1)} \ C_{k+1}^{\l}( \langle x' , y' \rangle) \\
 &\quad- \delta_{k,\ell} \frac{\l}{k+\l+1} (\ux' \wedge \uy') \ C_k^{\l+1}(\langle x' , y' \rangle).
\end{align*}
\end{lemma}

We can now show the following result.
\begin{theorem}\label{key}
For $x, y \in \mathbb{R}^m$ and $r \in \mathbb{R}_+$, one has
\begin{align*}
&\int_{\mS^{m-1}} \widetilde{K(r \eta, x)} \ K(y, r \eta) \ d \omega (\eta) \\
&= 2^{\l} \Gamma(\l+1) u^{-\l} J_{\l}(u) \ e^{-\frac{i}{2} (\cot \alpha)(|x|^2-|y|^2)} 
\end{align*}
with $u = \frac{r}{\sin{\alpha}} \sqrt{|x|^2+|y|^2-2 \langle x, y \rangle} =  \frac{r}{\sin{\alpha}} |x - y|$ and $\lambda=(m-2)/2$.
\end{theorem}
\begin{proof}
First we slightly rewrite $K(x,y)$ (see Section \ref{gentran}) and $\widetilde{K(x,y)}$ (see Theorem \ref{inverse}) as
\begin{align*}
K(x,y) &= \left( F_{\l}(w, \widetilde{z}) + (\ux' \wedge \uy') G_{\l}(w, \widetilde{z})  \right)\\
& \quad \times e^{\frac{i}{2}( \cot \alpha) (|x|^2 + |y|^2)}\\
\widetilde{K(x,y)} &= \left( S_{\l}(w, \widetilde{z}) + (\ux' \wedge \uy')  T_{\l}(w, \widetilde{z})  \right)\\
& \quad \times  e^{-\frac{i}{2}( \cot \alpha) (|x|^2 + |y|^2)} 
\end{align*}
with
\begin{align*}
F_{\l}(w, \widetilde{z})  &= \sum_{k=0}^{+\infty} \alpha_{k} \ (\widetilde{z})^{-\l}J_{k+\l}(\widetilde{z}) C^{\l}_{k}(w)\\
G_{\l}(w, \widetilde{z})  &= \sin{\alpha} \sum_{k=1}^{+\infty} \beta_{k} \ (\widetilde{z})^{-\l}J_{k+\l}(\widetilde{z}) C^{\l+1}_{k-1}(w)\\
S_{\l}(w, \widetilde{z})  &= \sum_{k=0}^{+\infty} s_{k} \ (\widetilde{z})^{-\l}J_{k+\l}(\widetilde{z}) C^{\l}_{k}(w)\\
T_{\l}(w, \widetilde{z})  &= \sin{\alpha} \sum_{k=1}^{+\infty} t_{k} \ (\widetilde{z})^{-\l}J_{k+\l}(\widetilde{z}) C^{\l+1}_{k-1}(w)
\end{align*}
where $s_k = (\beta_k \ \sin{\alpha}   + \alpha_k)/N_k^{\lambda}$, $t_k = -\beta_k/N_k^{\lambda}$,
\begin{align*}
&N_k^{\lambda} = \frac{1}{2^{2\l} (\Gamma (\l +1))^2} \ \left(\frac{\l}{\l+k} \alpha_{k} - \sin{\alpha} \ \frac{k}{2(\l+k)} \beta_{k} \right) \\
& \quad \times \left(\frac{\l}{\l+k} \alpha_{k} + \sin{\alpha} \frac{k+2\l}{2(\l+k)} \beta_{k} \right)
\end{align*}
and $x' = x/|x|, y' =y/|y|, w = \langle x' , y' \rangle, \widetilde{z} = (|x||y|) /\sin{\alpha}$, $\lambda = (m-2)/2$.

Using these decompositions, we obtain
\begin{align*}
&\int_{\mS^{m-1}} \widetilde{K(r \eta, x)} \ K(y, r \eta) \ d \omega (\eta) \\
&= e^{-\frac{i}{2} (\cot \alpha)(|x|^2-|y|^2)} (I_1 + I_2 + I_3 + I_4), 
\end{align*}
where we calculate the 4 pieces $I_1,I_2,I_3$ and $I_4$ separately.

For $I_1$, we use the reproducing property of the spherical harmonics to obtain
\begin{align*}
I_1 &= \int_{\mS^{m-1}} S_{\l}(w_1,\widetilde{z}_1) \ F_{\l}(w_2,\widetilde{z}_2)  \ d \omega (\eta)\\
&= \sum_{k=0}^{\infty} \frac{\l}{\l+k} s_k \alpha_k (\widetilde{z}_1 \widetilde{z}_2)^{-\l}  J_{k+\l}(\widetilde{z}_1)  J_{k+\l}(\widetilde{z}_2)  C_k^{\l}(\langle x',y'  \rangle),
\end{align*}
where we use the notations $\widetilde{z}_1 = (r |x|)/\sin{\alpha}$, $\widetilde{z}_2 = (r |y|)/\sin{\alpha}$, $w_1 =\langle \eta, x' \rangle$ and $w_2 =\langle y' ,\eta \rangle$. For $I_2$, we use Lemma \ref{lemma1} yielding 
\begin{align*}
I_2 &= \int_{\mS^{m-1}} (\underline{\eta} \wedge \ux') \ T_{\l}(w_1,\widetilde{z}_1) \ F_{\l}(w_2,\widetilde{z}_2)  \ d \omega (\eta)\\
&= - \sin{\alpha} \ (\ux' \wedge \uy') \  \sum_{k=1}^{\infty} \frac{\l}{\l+k} t_k \alpha_k (\widetilde{z}_1 \widetilde{z}_2)^{-\l}\\
& \quad \times J_{k+\l}(\widetilde{z}_1)  J_{k+\l}(\widetilde{z}_2) C_{k-1}^{\l+1}(\langle x',y'  \rangle),
\end{align*}
and similarly for $I_3$
\begin{align*}
I_3 &= \int_{\mS^{m-1}}  S_{\l}(w_1,\widetilde{z}_1) \ (\uy' \wedge \underline{\eta}) \ G_{\l}(w_2,\widetilde{z}_2)  \ d \omega (\eta)\\
&= - \sin{\alpha} \ (\ux' \wedge \uy') \  \sum_{k=1}^{\infty} \frac{\l}{\l+k} s_k \beta_k (\widetilde{z}_1 \widetilde{z}_2)^{-\l}  \\
& \quad \times J_{k+\l}(\widetilde{z}_1) J_{k+\l}(\widetilde{z}_2)  C_{k-1}^{\l+1}(\langle x',y'  \rangle).
\end{align*}
Finally, we can calculate the term $I_4$ using Lemma \ref{lemma2} as follows
\begin{align*}
I_4 &= \int_{\mS^{m-1}}  (\underline{\eta} \wedge \ux') \ T_{\l}(w_1,\widetilde{z}_1) \ (\uy' \wedge \underline{\eta}) \ G_{\l}(w_2,\widetilde{z}_2)  \ d \omega (\eta)\\
&= (\sin{\alpha})^2 \sum_{k=1}^{\infty} \frac{k(k+2\l)}{4\l(k+\l)} t_k \beta_k (\widetilde{z}_1 \widetilde{z}_2)^{-\l} \\
& \qquad \times J_{k+\l}(\widetilde{z}_1)  J_{k+\l}(\widetilde{z}_2) C_{k}^{\l}(\langle x',y'  \rangle)\\
&-  (\sin{\alpha})^2 (\ux' \wedge \uy')   \sum_{k=1}^{\infty} \frac{\l}{\l+k} t_k \beta_k (\widetilde{z}_1 \widetilde{z}_2)^{-\l} \\
& \qquad \times J_{k+\l}(\widetilde{z}_1)  J_{k+\l}(\widetilde{z}_2)  C_{k-1}^{\l+1}(\langle x',y'  \rangle).
\end{align*}
Adding these 4 terms then gives
\begin{align*}
& I_1+I_2+I_3+I_4\\
&= \sum_{k=0}^{\infty} \left( \frac{\l}{\l+k} s_k \alpha_k + (\sin{\alpha})^2 \frac{k(k+2\l)}{4\l (k+\l)} t_k \beta_k \right)  (\widetilde{z}_1 \widetilde{z}_2)^{-\l}  \\
& \qquad \times J_{k+\l}(\widetilde{z}_1)  J_{k+\l}(\widetilde{z}_2) C_{k}^{\l}(\langle x',y'  \rangle)\\
& - (\ux' \wedge \uy') \sin{\alpha} \sum_{k=1}^{\infty} \frac{\l}{\l+k} (t_k \alpha_k + s_k \beta_k + \sin{\alpha} \ t_k \beta_k) \\
& \qquad \times (\widetilde{z}_1 \widetilde{z}_2)^{-\l} J_{k+\l}(\widetilde{z}_1)  J_{k+\l}(\widetilde{z}_2) C_{k-1}^{\l+1}(\langle x',y'  \rangle).
\end{align*}
It is easy to check for all $k$ that $ t_k \alpha_k + s_k \beta_k + \sin{\alpha} \ t_k \beta_k = 0$, so the term in $(\ux' \wedge \uy')$ vanishes. Similarly we can compute that
\begin{align*}
&\frac{\l}{\l+k} s_k \alpha_k + (\sin{\alpha})^2 \frac{k(k+2\l)}{4\l (k+\l)} t_k \beta_k\\
& = 2^{2\l} (\Gamma(\l+1))^2 \frac{1}{\l} (\l+k),
\end{align*}
hence we conclude that
\begin{align*}
&I_1+I_2+I_3+I_4\\
& = 2^{2\l} (\Gamma(\l+1))^2 \frac{1}{\l} \sum_{k=0}^{\infty} (\l+k) (\widetilde{z}_1 \widetilde{z}_2)^{-\l} \\
& \quad \times J_{k+\l}(\widetilde{z}_1) J_{k+\l}(\widetilde{z}_2)  C_{k}^{\l}(\langle x',y'  \rangle).
\end{align*}
Now we invoke the addition formula for Bessel functions (see e.g. \cite{Magnus}, p. 107) yielding
\begin{align*}
u^{-\l} J_{\l}(u) &= 2^{\l} \Gamma(\l) \sum_{k=0}^{\infty} (\l+k) (\widetilde{z}_1 \widetilde{z}_2)^{-\l} \\
& \quad \times J_{k+\l}(\widetilde{z}_1) J_{k+\l}(\widetilde{z}_2)  C_{k}^{\l}(\langle x',y'  \rangle)
\end{align*}
with $u = \frac{r}{\sin{\alpha}} \sqrt{|x|^2+|y|^2 - 2 \langle x, y \rangle}$. This completes the proof of the theorem.
 \end{proof}

We can now prove the main theorem in this section.
\begin{theorem}\label{transradial}
Let $f \in \cS(\mR^{m})$ be a real-valued radial function on $\mathbb{R}^m$, i.e. $f(x) = f_0(|x|)$ with $f_0 : \mathbb{R}_+ \longmapsto \mathbb{R}$, then
\begin{align*}
\tau_{y}^K f(x) &= \frac{2 \alpha_0}{\Gamma(\l + 1)} (1-e^{2i\alpha})^{-m/2} e^{-\frac{i}{2} (\cot \alpha)(|x|^2-|y|^2)} \\
& \quad \times H_{\lambda} \left\lbrack \cF_{\alpha} \left( f \right) \right\rbrack \left( \frac{|x - y|}{\sin{\alpha}} \right) 
\end{align*}
with $\l = (m-2)/2$, $\cF_{\alpha}$ the fractional version of the classical Fourier transform given by the integral transform
\begin{align*}
\cF_{\alpha}\left( f \right) (y) = \rho_{\alpha,m} \int_{\mathbb{R}^m} e^{-\frac{i \left< x, y \right>}{\sin{\alpha}}} e^{\frac{i}{2} \cot{\alpha}(|x|^2+|y|^2)} f(x) dx
\end{align*}
and $H_{\lambda}$ the Hankel transform defined by
\begin{displaymath}
H_{\lambda}f(s) := \int_{0}^{\infty} f(r) \ \frac{J_{\lambda}(rs)}{(rs)^{\lambda}} r^{2\lambda +1} dr.
\end{displaymath}
\end{theorem}
\begin{proof}
If $f(x) = f_0(|x|)$ is real-valued and radial, then by means of Proposition \ref{radbehavior1} we obtain that $\cF_K \left( f \right) $ is a radial function as well and it coincides (up to a factor) with the fractional version $\cF_{\alpha}$ of the ordinary Fourier transform. More precisely we have
\begin{displaymath}
\cF_K \left( f \right)(x)= \frac{2^{-\l} \alpha_0}{\Gamma(\l +1)} \ \cF_{\alpha}\left( f \right) (x).
\end{displaymath}
Hence, by definition of $\tau^K_{y}$ we obtain
\begin{displaymath}
\cF_K ( \tau^K_{y} f )(x)= \frac{2^{-\l} \alpha_0}{\Gamma (\lambda +1)} \ K(y,x) \ \cF_{\alpha} \left( f \right) (x).
\end{displaymath}
Taking the inverse and using polar coordinates $x = r \eta$, $r=|x|$, we obtain
\begin{align*}
\tau^K_{y} f (x') &= \frac{2^{-\l} \alpha_0}{\Gamma (\lambda +1)} (\pi(1-e^{2i\alpha}))^{-m/2} \int_{\mathbb{R}^m} \widetilde{K(x,x')} \\
& \quad \times K(y,x) \ \cF_{\alpha}\left( f \right) (|x|) \ dx\\
&= \frac{2^{-\l+1} \alpha_0}{(\Gamma (\lambda +1))^2} (1-e^{2i\alpha})^{-m/2} \int_0^{+\infty} \cF_{\alpha}\left( f \right) (r) \\
& \times \left( \int_{\mS^{m-1}} \widetilde{K(r \eta,x')} \ K(y,r\eta) d\omega(\eta) \right)  r^{m-1} dr.
\end{align*}
In view of Theorem \ref{key} this becomes
\begin{align*}
\tau^K_{y} f (x') &= \frac{2 \alpha_0}{\Gamma(\l+1)} \ (1-e^{2i\alpha})^{-m/2} e^{-\frac{i}{2} (\cot \alpha)(|x'|^2-|y|^2)}  \\
& \times \left(\int_0^{+\infty} \cF_{\alpha}\left( f \right) (r) \ r^{m-1} u^{-\l} J_{\l}(u) dr \right) 
\end{align*}
with $u= \frac{r}{\sin{\alpha}} |x'-y|$. Taking into account the definition of the Hankel transform $H_{\l}$, we finally obtain
\begin{align*}
&\tau^K_{y}f(x') = \frac{2 \alpha_0}{\Gamma(\l + 1)} (1-e^{2i\alpha})^{-m/2} e^{-\frac{i}{2} (\cot \alpha)(|x'|^2-|y|^2)}\\
&\quad \times   H_{\lambda} \left\lbrack \cF_{\alpha} \left( f \right) \right\rbrack \left( \frac{|x' - y|}{\sin{\alpha}} \right).
\end{align*}
 \end{proof}

In the special case when $\alpha = \pi/2$, it follows from Theorem \ref{transradial} that the generalized translation operator $\tau^K_{y}$ coincides, up to a constant, with geometric translation if $f$ is a radial function.
\begin{corollary}\label{specialcasetransl}
Let $f \in \cS(\mR^{m})$ be a real-valued radial function on $\mathbb{R}^m$, i.e. $f(x) = f_0(|x|)$ with $f_0 : \mathbb{R}_+ \longmapsto \mathbb{R}$, then in case of $\alpha = \pi/2$, one has
\begin{displaymath}
\tau^K_{y}f(x) = \frac{2^{-\l} \alpha_0}{\Gamma(\l+1) }  f(|x - y|)
\end{displaymath}
with $\l = (m-2)/2$.
\end{corollary}
\begin{proof}
In case of $\alpha = \pi/2$, the fractional Fourier transform $\cF_{\alpha}$ reduces to the classical Fourier transform $\cF$.
Moreover, taking into account that for radial functions the classical Fourier transform coincides with the Hankel transform $H_{\l}$ and that the inverse Hankel transform is given by
\begin{displaymath}
f(s):= \int_0^{+\infty} H_{\l}f(r) \ \frac{J_{\l}(rs)}{(rs)^{\l}} r^{2\l+1} dr,
\end{displaymath}
which holds under mild conditions on $f$, the desired result follows.
 \end{proof}

These results allow us to give more details about the new convolution products, when one of the functions involved is radial. This is summarized in the following proposition.
\begin{proposition}
If $g \in \cS(\mR^{m}) \otimes \cC l_{0,m}$ and $f \in \cS(\mR^{m})$ is a real-valued radial function, then
\begin{align*}
&\cF_K \left( f \ast_{L} g \right) = \cF_K \left( g \ast_L f \right)\\
& = \cF_K \left( f \ast_R g \right) = \cF_K \left( g \ast_R f \right)\\
 &=\rho_{\a,m}^{-1}   \cF_K \left( f \right) \cF_K \left( g \right).
\end{align*}
In particular, under these assumptions one has
\begin{displaymath}
 f \ast_{L} g  =  g \ast_L f = f \ast_R g =  g \ast_R f.
 \end{displaymath}
\end{proposition}
\begin{proof} If $f$ is a real-valued radial function, then from Proposition \ref{radbehavior1} we obtain that $\cF_K \left( f \right)$ is a real-valued radial function as well. This implies that $\cF_K \left( f \right) \, h = h\,  \cF_K \left( f \right)$ for any Clifford algebra-valued function $h$. Moreover, from Theorem \ref{transradial} we observe that also $\tau^K_{y}f$ commutes with any Clifford algebra-valued function.

Let us now, for example, show that $\cF_K \left( f \ast_{L} g \right) = \rho_{\a,m}^{-1}  \cF_K \left( f \right) \cF_K \left( g \right)$. The other statements are proved similarly. By definition of the transform $\cF_K$, the convolution $\ast_{L}$ and the Fubini theorem, we obtain consecutively
\begin{eqnarray*}  
& & \cF_K \left( f \ast_{L} g \right) (x)\\
 & = &\rho_{\a,m} \int_{\mathbb{R}^m} K(t,x) \, \left( \int_{\mathbb{R}^m} [\tau^K_{y}f(t)]  g(y) dy \right) \, dt\\
 & = &  \int_{\mathbb{R}^m} \left( \rho_{\a,m} \int_{\mathbb{R}^m} K(t,x) [ \tau^K_{y}f(t) ] dt \right) \, g(y) dy\\
& = &  \int_{\mathbb{R}^m} \cF_K \left(  \tau^K_{y}f \right)(x)\ g(y) dy\\
& = &   \int_{\mathbb{R}^m} K(y,x) \ \cF_K \left( f \right)(x)g(y) dy\\
& = &\rho_{\a,m}^{-1} \cF_K \left( f \right)(x) \rho_{\a,m} \int_{\mathbb{R}^m} K(y,x) g(y) dy\\
& = & \rho_{\a,m}^{-1}\cF_K \left( f \right)(x) \,\cF_K \left( g \right) (x),
\end{eqnarray*}
where we have used Definition \ref{defTrans} and the fact that\\ $K(y,x) \cF_K\left( f \right)= \cF_K\left( f \right) K(y,x)$.

Finally, by applying the inverse transform $\cF_K^{-1}$ on
\begin{displaymath}
\cF_K \left( f \ast_{L} g \right) = \cF_K \left( g \ast_L f \right) = \cF_K \left( f \ast_R g \right) = \cF_K \left( g \ast_R f \right),
\end{displaymath}
we obtain
\begin{displaymath}
f \ast_{L} g  =  g \ast_L f = f \ast_R g =  g \ast_R f.
\end{displaymath}
 
\end{proof}


\section{Convolution products in Approach B}
\label{secB}

\subsection{Definition, eigenvalues and eigenfunctions}

Let us start by defining the family of transforms we will be looking at in this section.

\begin{definition}\label{def:gft}
Denote by $\cI_{m}$ the set $\{ i \in \cC l_{0,m} | i^{2}=-1\}$ of geometric square roots of minus one. Let $F_1:=\{i_1,...,i_\mu\},F_2:=\{i_{\mu+1},...,i_m\}$ be two ordered finite sets of such square roots, $i_k\in\mathcal I_m,\forall k=1,...,m$. The geometric Fourier transform (GFT) $\mathcal F_{F_1,F_2}$ of a function $f: \mR^{m} \rightarrow \cC l_{0,m}$ takes the form:
\begin{align*}
\cF_{F_1,F_2} (f)(u)&:= (2\pi)^{-\frac m2}\int_{\mR^{m}} \left( \prod_{k=1}^{\mu} e^{-i_{k} x_{k}u_{k}} \right)\\
& \quad \times f(x)  \left(\prod_{k=\mu+1}^{m} e^{-i_{k} x_{k}u_{k}} \right) dx.
\end{align*}
\end{definition}
This definition is a special case of the general geometric Fourier transforms from \cite{BU}, where also non-linear functions in the exponentials are allowed. We will use this restriction to guarantee that the inverse transform of any GFT is a GFT itself, namely 
\[
 \cF^{-1}_{F_1,F_2}=\cF_{\{-i_\mu,...,-i_1\},\{-i_m,...,-i_{\mu+1}\}}.
\]
Moreover, the restriction also allows us to obtain the eigenvalues and eigenfunctions of the GFT. They are given in the following theorem.

\begin{theorem}
\label{eigenvalsGFT}
The basis $\{ \psi_{j_1,j_2, \ldots, j_m}\}$ of $\cS(\mR^{m}) \otimes \cC l_{0,m}$ diagonalizes the GFT. One has
\begin{align*}
&\cF_{F_1,F_2} ( \psi_{j_1,j_2, \ldots, j_m})\\
& = \left( \prod_{k=1}^{\mu}(-i_{k})^{j_k} \right)  \psi_{j_1,j_2, \ldots, j_m}
 \left(\prod_{k=\mu+1}^{m} (-i_{k})^{j_k} \right).
 \end{align*}
\end{theorem}

\begin{proof}
By direct computation, we find
\begin{align*}
 &(2\pi)^{\frac m2} \cF_{F_1,F_2} ( \psi_{j_1,j_2, \ldots, j_m})\\
  &= \int_{\mR^{m}} \left( \prod_{k=1}^{\mu} e^{-i_{k} x_{k}u_{k}} \right) \\
  & \quad \times \psi_{j_1,j_2, \ldots, j_m}(x)  \left(\prod_{k=\mu+1}^{m} e^{-i_{k} x_{k}u_{k}} \right) dx\\
  &= \left( \prod_{k=1}^{\mu}\int_{\mR}  e^{-i_{k} x_{k}u_{k}} \psi_{j_{k}}(x_{k}) dx_{k} \right) \\
  & \quad \times \left( \prod_{k=\mu+1}^{m} \int_{\mR}  \psi_{j_{k}}(x_{k})e^{-i_{k} x_{k}u_{k}}dx_{k} \right) \end{align*}\begin{align*}
&=  \left( \prod_{k=1}^{\mu}(-i_{k})^{j_k} \psi_{j_{k}}(u_{k}) \right) \left( \prod_{k=\mu+1}^{m}   \psi_{j_{k}}(u_{k})(-i_{k})^{j_k}\right) \\
&= \left( \prod_{k=1}^{\mu}(-i_{k})^{j_k} \right)  \psi_{j_1,j_2, \ldots, j_m}
 \left(\prod_{k=\mu+1}^{m} (-i_{k})^{j_k} \right).
\end{align*}
Here, we used the result
\[
\int_{\mR}  \psi_{j_{k}}(x_{k})e^{-i_{k} x_{k}u_{k}}dx_{k} = (-i_{k})^{j_k} \psi_{j_{k}}(u_{k}) 
\]
which is a special case of formula (\ref{eigCLFT}).
 
\end{proof}

An important example of this family of transforms is the two-sided quaternionic Fourier transform (qFT). We will treat it in detail in Section \ref{secQFT}

\begin{remark}
Note that all subsequent results also hold in the more general setup of functions $f: \mR^{m} \rightarrow \cC l_{p,q}$, compare \cite{BU}. However, this is not the case for the transforms of approach \textbf{A}, due to the fact that the Dirac operator is no longer elliptic in arbitrary signature.
\end{remark}

\subsection{Convolution formula based on Mustard's idea}
\label{GFTMustard}

Based on the idea of Mustard we define a generalized convolution for any geometric Fourier transform.
\begin{definition}\label{def:conv_must}
For any GFT $\cF_{F_1,F_2}$ we define the convolution $*_{F_1,F_2}$ by
\[
 (f*_{F_1,F_2}g)(x):=(2\pi)^{\frac m 2}\cF_{F_1,F_2}^{-1}(\cF_{F_1,F_2}(f)\cF_{F_1,F_2}(g))(x).
\]
\end{definition}
Now we want to express the convolution $*_{F_1,F_2}$ by means of the standard convolution
\[
 (f* g)(x)=\int_{\mR^m}f(y) g(x-y) \ dy.
\]
To that aim, we introduce the following notation.
\begin{notation}\label{not:phi,gamma}
For functions $f,g:\mR^m\to\cC l_{0,m}$ and multi-indices $\vec{\phi},\vec{\gamma} \in\{0,1\}^{m}$ we put
\begin{align*}
 f^{\vec{\phi}}(x)&:=f((-1)^{\phi_{1}}x_1,...,(-1)^{\phi_{m}}x_m),\\
g^{\vec{\gamma}}(x)&:=g((-1)^{\gamma_{1}}x_1,...,(-1)^{\gamma_{m}}x_m).
\end{align*}
\end{notation}

The following theorem is our main result. 
\begin{theorem}
\label{mustardconvGFT}
Let $J=\{0,1\}^{4\times m}$ with $j_{1,k}+j_{2,k}+j_{3,k}\in\{0,2\}$ and $j_{4,k}=0$ for all $k=1,\ldots,m$ be a set of multi-indices. Any generalized convolution $*_{F_1,F_2}$ from Definition \ref{def:conv_must} can be expressed as a sum of classical convolutions using Notation \ref{not:phi,gamma} by
 \begin{align*}
& (f*_{F_1,F_2}g)(x)\\
&=\frac1{4^m}\sum_{\vec{j} \in J}\sum_{\vec{\phi},\vec{\gamma} \in\{0,1\}^m}c_{\vec{j},\vec{\phi},\vec{\gamma}}\\
&  \left(\prod_{k=\mu}^1(i_k)^{j_{1,k}}\prod_{k=1}^\mu (-i_k)^{j_{2,k}}f^{\vec{\phi}}\prod_{k=\mu+1}^m(-i_k)^{j_{2,k}} \right) \\
  & \quad  * \left( \prod_{k=1}^\mu (-i_k)^{j_{3,k}} g^{\vec{\gamma}} \prod_{k=\mu+1}^m (-i_k)^{j_{3,k}}\prod_{k=m}^{\mu+1}(i_k)^{j_{1,k}} \right) (x)
\end{align*}
with the sign $c_{\vec{j},\vec{\phi},\vec{\gamma}}$ given by
\[
c_{\vec{j},\vec{\phi},\vec{\gamma}}=\prod_{k=1}^m(-1)^{(j_{(2{\phi}_k +{\gamma}_k+1,k)}+1)(\delta_{(j_{1,k}+j_{2,k}+j_{3,k})}-1)},
\]
where
\begin{equation*}
\delta_{(\ell)} := \begin{cases}
                            1 ,&\text{ if } \ell =0,\\
			    0 ,&\text{ if } \ell \not= 0.
                           \end{cases}
\end{equation*}
\end{theorem}

\begin{proof}
For the multi-indices $ \vec{j} \in \{0,1\}^{3\times m}$, which we address by $j_{\nu,k} \in \lbrace 0,1 \rbrace$ for $\nu=1,2,3$ representing the integration variables $x,y,z$ and $k=1,...,m$ representing the coordinates, we use the notation
\begin{equation*}
e^{-i_k x_k u_k}_{j_{1,k}} :=\begin{cases}
                            \cos(x_k u_k),&\text{ if }j_{1,k}=0,\\
			    -\sin(x_k u_k),&\text{ if }j_{1,k}=1
                           \end{cases}
\end{equation*}
and for $\nu=2,3$ corresponding to $y,z$ analogously. Note that $e^{-i_k x_k u_k}_{j_{1,k}}$ is always real valued and therefore in the center of the Clifford algebra. We get
\begin{align*}
& (2 \pi)^{m}(f \ast_{F_1,F_2} g)(x)\\
&=  \int_{\mR^m} \prod_{k=\mu}^1 e^{i_k x_k u_k} \\
 & \quad \times \left( \int_{\mR^m}\prod_{k=1}^\mu e^{-i_k y_k u_k} f(y) \prod_{k=\mu+1}^m e^{-i_k y_k u_k} dy \right)\\
& \quad \times \left( \int_{\mR^m} \prod_{k=1}^\mu e^{-i_k z_k u_k} g(z) \prod_{k=\mu+1}^m e^{-i_k z_k u_k} dz \right) \\
 & \quad \times\prod_{k=m}^{\mu+1} e^{i_k x_k u_k} du\\ 
&=   \int_{\mR^{3m}} \prod_{k=\mu}^1 e^{i_k x_k u_k} \prod_{k=1}^\mu e^{-i_k y_k u_k} f(y) \prod_{k=\mu+1}^m e^{-i_k y_k u_k}\\
 & \quad \times \prod_{k=1}^\mu e^{-i_k z_k u_k} g(z) \prod_{k=\mu+1}^m e^{-i_k z_k u_k} \prod_{k=m}^{\mu+1} e^{i_k x_k u_k} \ dy \ dz \ du
 \end{align*}
 \begin{align*}
&=  \sum_{\vec{j} \in\{0,1\}^{3\times m}} \int_{\mR^{3m}}
\left( \prod_{k=\mu}^1 (i_k)^{j_{1,k}} \prod_{k=\mu}^1 e^{i_k x_k u_k}_{j_{1,k}} \right)\\
 & \quad \times  \left( \prod_{k=1}^\mu (i_k)^{j_{2,k}} \prod_{k=1}^\mu e^{- i_k y_k u_k}_{j_{2,k}} \right)  f(y) \\
 & \quad \times
\left( \prod_{k=\mu+1}^m (i_k)^{j_{2,k}} \prod_{k=\mu+1}^m e^{-i_k y_k u_k}_{j_{2,k}} \right) \\
 & \quad \times\left( \prod_{k=1}^\mu (i_k)^{j_{3,k}} \prod_{k=1}^\mu e^{-i_k z_k u_k}_{j_{3,k}} \right)\\
 & \quad \times
g(z)  \left(\prod_{k=\mu+1}^m (i_k)^{j_{3,k}} \prod_{k=\mu+1}^m e^{-i_k z_k u_k}_{j_{3,k}} \right) \\
 & \quad \times \left( \prod_{k=m}^{\mu+1} (i_k)^{j_{1,k}} \prod_{k=m}^{\mu+1} e^{i_k x_k u_k}_{j_{1,k}} \right) \ dy \ dz \ du\\
&=\sum_{\vec{j}\in\{0,1\}^{3\times m}} \int_{\mR^{2m}} \left( \prod_{k=1}^m \int_{- \infty}^{+ \infty} e^{i_k x_k u_k}_{j_{1,k}} e^{-i_k y_k u_k}_{j_{2,k}}  e^{-i_k z_k u_k}_{j_{3,k}} \ du_k \right) \\
 & \quad \times \left(  \prod_{k=\mu}^1 (i_k)^{j_{1,k}} \right) \left( \prod_{k=1}^\mu (i_k)^{j_{2,k}} \right) f(y)  \left( \prod_{k=\mu+1}^m (i_k)^{j_{2,k}} \right)  \\
 & \quad \times \left( \prod_{k=1}^\mu (i_k)^{j_{3,k}}  \right)  g(z)    \left(\prod_{k=\mu+1}^m (i_k)^{j_{3,k}} \right)  \\
 & \quad \times \left( \prod_{k=m}^{\mu+1} (i_k)^{j_{1,k}} \right) dy \ dz. 
\end{align*}
In order to calculate the integration over the $u$-variable, we use the trigonometric equalities
\begin{align}
\label{trig}
\begin{split}
\sin x  \sin y  \sin z &= \frac{1}{4} \Big(\sin (x+y-z) + \sin (y+z-x) \\
& \quad + \sin (z+x-y) - \sin (x+y+z)\Big)\\
\cos x  \cos y  \cos z &= \frac{1}{4} \Big(\cos (x+y-z) + \cos (y+z-x) \\
& \quad+ \cos (z+x-y) + \cos (x+y+z)\Big)\\
\sin x  \sin y  \cos z &= \frac{1}{4} \Big(- \cos (x+y-z) + \cos (y+z-x)\\
& \quad + \cos (z+x-y) - \cos (x+y+z)\Big)\\
\sin x  \cos y  \cos z &= \frac{1}{4} \Big(\sin (x+y-z) - \sin (y+z-x) \\
& \quad+ \sin (z+x-y) + \sin (x+y+z)\Big).
\end{split}
\end{align}
For each index $k=1,...,m$ and $j_{1,k}+j_{2,k}+j_{3,k}=0$ we have that the $u_k$-integrand takes the form 
\begin{align*}
&e^{i_k x_k u_k}_{j_{1,k}} e^{-i_k y_k u_k}_{j_{2,k}} e^{-i_k z_k u_k}_{j_{3,k}}\\
&= \frac{1}{4} \biggl\lbrace \cos ((-x_k+y_k+z_k)u_k)+  \cos ((x_k-y_k+z_k)u_k) \\ 
& +\cos ((x_k+y_k-z_k)u_k)+\cos ((x_k+y_k+z_k)u_k) \biggr\rbrace
\end{align*}
and for $j_{1,k}+j_{2,k}+j_{3,k}=2$ we have
\begin{align*}
& e^{i_k x_k u_k}_{j_{1,k}} e^{-i_k y_k u_k}_{j_{2,k}} e^{-i_k z_k u_k}_{j_{3,k}}\\
&= \frac{(-1)^{j_{2,k} + j_{3,k}}}{4} \ \biggl\lbrace (-1)^{(j_{1,k}+1)} \cos ((-x_k+y_k+z_k)u_k) \\
& \quad +  (-1)^{(j_{2,k}+1)} \cos ((x_k-y_k+z_k)u_k) \\
& \quad+(-1)^{(j_{3,k}+1)} \cos ((x_k+y_k-z_k)u_k) \\
& \quad- \cos ((x_k+y_k+z_k)u_k) \biggr\rbrace.
\end{align*}
Hence for $j_{1,k}+j_{2,k}+j_{3,k}$ even we can summarize both in
\begin{equation*}
\begin{aligned}
& e^{i_k x_k u_k}_{j_{1,k}} e^{-i_k y_k u_k}_{j_{2,k}} e^{-i_k z_k u_k}_{j_{3,k}}= \frac{(-1)^{j_{2,k} + j_{3,k}}}{4} \\
& \times \biggl\lbrace (-1)^{(j_{1,k}+1)(\delta_{(j_{1,k}+j_{2,k}+j_{3,k})}-1)} \cos ((-x_k+y_k+z_k)u_k)\\
&+  (-1)^{(j_{2,k}+1)(\delta_{(j_{1,k}+j_{2,k}+j_{3,k})}-1)} \cos ((x_k-y_k+z_k)u_k) \\
& +(-1)^{(j_{3,k}+1)(\delta_{(j_{1,k}+j_{2,k}+j_{3,k})}-1)} \cos ((x_k+y_k-z_k)u_k)\\
& +(-1)^{(\delta_{(j_{1,k}+j_{2,k}+j_{3,k})}-1)}\cos ((x_k+y_k+z_k)u_k) \biggr\rbrace
\end{aligned}
\end{equation*}
with
\begin{equation*}
\delta_{(\ell)} := \begin{cases}
                            1 ,&\text{ if } \ell =0,\\
			    0 ,&\text{ if } \ell \not= 0.
                           \end{cases}
\end{equation*} 
In case of $j_{1,k}+j_{2,k}+j_{3,k}$ odd, the trigonometric equations (\ref{trig}) immediately show that $e^{i_k x_k u_k}_{j_{1,k}} e^{-i_k y_k u_k}_{j_{2,k}} e^{-i_k z_k u_k}_{j_{3,k}}$ equals a sum of sine functions.

Hence, by splitting the equation
\begin{displaymath}
\delta(x)=\frac{1}{2\pi}\int_{\mR} e^{-i xu} d u
\end{displaymath}
into real and imaginary part:
\begin{equation*}
\begin{aligned}
\delta(x)&=\frac{1}{2\pi}\int_{\mR} \cos( xu) du,\\
0&=\frac{1}{2\pi}\int_{\mR} \sin( xu) \ du,
\end{aligned}
\end{equation*}
one sees that all summands with $j_{1,k}+j_{2,k}+j_{3,k}$ odd vanish after integration over $u_k$, while the even ones become delta distributions. As a result we have
\begin{equation*}
\begin{aligned}
& \int_{- \infty}^{+ \infty} e^{i_k x_k u_k}_{j_{1,k}} e^{-i_k y_k u_k}_{j_{2,k}} e^{-i_k z_k u_k}_{j_{3,k}} \ du_k\\
&= \begin{cases}
0 , \quad \text{ if } j_{1,k}+j_{2,k}+j_{3,k} \mathrm{\ is \ odd},\\
\frac{\pi}{2} (-1)^{j_{2,k} + j_{3,k}} \sum_{\phi_k, \gamma_k \in \lbrace 0,1 \rbrace}\\ \quad \times (-1)^{(j_{(2\phi_k + \gamma_k+1,k)}+1)(\delta_{(j_{1,k}+j_{2,k}+j_{3,k})}-1)} &\\
\quad \times   \delta(x_k + (-1)^{\phi_k+1} y_k + (-1)^{\gamma_k+1} z_k),\\ \quad\text{ if } j_{1,k}+j_{2,k}+j_{3,k} \mathrm{\ is \ even}.
\end{cases}
\end{aligned}
\end{equation*}
with $j_{4,k} = 0 \ \forall k$ and also
\begin{equation*}
\begin{aligned}
& (f \ast_{F_1,F_2} g)(x)\\
 &= \frac{1}{4^m} \sum_{\stackrel{\vec{j}\in\{0,1\}^{4 \times m},}{\forall k=1,...,m:j_{1,k}+j_{2,k}+j_{3,k}\in\{0,2\}, j_{4,k}=0}} \sum_{\vec{\phi},\vec{\gamma} \in\{0,1\}^m} \\
 & \quad \times \left( \prod_{k=1}^{m}  (-1)^{(j_{(2\phi_k + \gamma_k+1,k)}+1)(\delta_{(j_{1,k}+j_{2,k}+j_{3,k})}-1)} \right)\\
& \quad \times \int_{\mR^m}  \left(  \prod_{k=\mu}^1 (i_k)^{j_{1,k}} \right) \left( \prod_{k=1}^\mu (-i_k)^{j_{2,k}} \right)  f(y) \\
 & \quad \times \left( \prod_{k=\mu+1}^m (-i_k)^{j_{2,k}} \right)  \left( \prod_{k=1}^\mu (-i_k)^{j_{3,k}}  \right) \\
& \quad \times \left( \int_{\mR^m} g(z_1, \ldots , z_m)  \right.\\  &\quad \times \left. \prod_{k=1}^{m}  \delta(x_k + (-1)^{\phi_k+1} y_k + (-1)^{\gamma_k+1} z_k) dz \right) \\
 & \quad \times \left(   \prod_{k=\mu+1}^m (-i_k)^{j_{3,k}} \right) \left( \prod_{k=m}^{\mu+1} (i_k)^{j_{1,k}} \right) dy. 
\end{aligned}
\end{equation*}
For each index $k=1,...,m$ using
\begin{displaymath}
g(x)=\int_{\mR} g(z) \delta(z-x) \ dz,
\end{displaymath}
we obtain
\begin{equation*}
\begin{aligned}
& (f \ast_{F_1,F_2} g)(x)\\
 &= \frac{1}{4^m} \sum_{\stackrel{\vec{j}\in\{0,1\}^{4\times m},}{\forall k=1,...,m:j_{1,k}+j_{2,k}+j_{3,k}\in\{0,2\}, j_{4,k}=0}} \sum_{\vec{\phi},\vec{\gamma} \in\{0,1\}^m} \\
 & \quad \times  \left( \prod_{k=1}^{m}  (-1)^{(j_{(2\phi_k + \gamma_k+1,k)}+1)(\delta_{(j_{1,k}+j_{2,k}+j_{3,k})}-1)} \right)\\
& \int_{\mR^m}  \left(  \prod_{k=\mu}^1 (i_k)^{j_{1,k}} \right) \left( \prod_{k=1}^\mu (-i_k)^{j_{2,k}} \right)  f(y) \\
 & \quad \times \left( \prod_{k=\mu+1}^m (-i_k)^{j_{2,k}} \right)  \left( \prod_{k=1}^\mu (-i_k)^{j_{3,k}}  \right) \\
& g((-1)^{\gamma_1} (x_1-(-1)^{\phi_1} y_1), \ldots , (-1)^{\gamma_m} (x_m-(-1)^{\phi_m} y_m) ) \\
 & \quad \times \left(   \prod_{k=\mu+1}^m (-i_k)^{j_{3,k}} \right)   \left( \prod_{k=m}^{\mu+1} (i_k)^{j_{1,k}} \right)  dy. 
\end{aligned}
\end{equation*}
Finally we execute the substitution $y'_k = (-1)^{\phi_k} y_k$, $k=1, \ldots,m$, and obtain
\begin{equation*}
\begin{aligned}
& (f \ast_{F_1,F_2} g)(x)\\
&= \frac{1}{4^m} \sum_{\stackrel{\vec{j}\in\{0,1\}^{4\times m},}{\forall k=1,...,m:j_{1,k}+j_{2,k}+j_{3,k}\in\{0,2\},j_{4,k}=0}} \sum_{\vec{\phi},\vec{\gamma} \in\{0,1\}^m}\\
& \quad \left( \prod_{k=1}^{m}  (-1)^{(j_{(2\phi_k + \gamma_k+1,k)}+1)(\delta_{(j_{1,k}+j_{2,k}+j_{3,k})}-1)} \right)\\
&  \left( \prod_{k=\mu}^1 (i_k)^{j_{1,k}}   \prod_{k=1}^\mu (-i_k)^{j_{2,k}}  f^{\vec{\phi}} \ \prod_{k=\mu+1}^m (-i_k)^{j_{2,k}} \right) \\
& \quad \ast  \left( \prod_{k=1}^\mu (-i_k)^{j_{3,k}}  g^{\vec{\gamma}}  \prod_{k=\mu+1}^m (-i_k)^{j_{3,k}} \prod_{k=m}^{\mu+1} (i_k)^{j_{1,k}}  \right) (x).
\end{aligned}
\end{equation*}
 
\end{proof}

\begin{remark}
One can now introduce an immediate analog of the convolution $f *_2 g$ mentioned in the introduction, by considering $g \ast_{F_1,F_2} f$. 
\end{remark}

\subsection{Convolution formula based on generalized translation}
\label{GFTtrans}

We define a generalized translation operator related to the GFT $\cF_{F_1,F_2}$. Contrary to our original definition in formula (\ref{gentrans}), we now have to take into account that the kernel consists of two parts.

\begin{definition}\label{def:transl}
For any GFT $\cF_{F_1,F_2}$ the general translation operator $\tau^{F_1,F_2}_{y}$ is defined by the relation
\begin{align*}
& \cF_{F_1,F_2}(\tau^{F_1,F_2}_{y}f)(u)\\
&:=\left(\prod_{k=1}^\mu e^{-i_ky_ku_k} \right)\cF_{F_1,F_2}(f)(u) \left(\prod_{k={\mu+1}}^m e^{-i_ky_ku_k} \right).
\end{align*}
\end{definition}

With calculations analogous to the ones in the previous section we can express the generalized translation operator $\tau^{F_1,F_2}_{y}$ by means of the standard translation
\[
 \tau_{y}f(x)=f(x-y).
\]

\begin{theorem}
\label{transGFT}
Let $J=\{0,1\}^{4\times m}$ with $j_{1,k}+j_{2,k}+j_{3,k}\in\{0,2\}$ and $j_{4,k}=0$ for all $k=1,\ldots,m$ be a set of multi-indices. The generalized translation operator $\tau^{F_1,F_2}_{y}$ from Definition \ref{def:transl} can be expressed as the sum of classical translations $\tau_{y^{\vec{\phi}}}f^{\vec{\gamma}}(x) = f^{\vec{\gamma}}(x-y^{\vec{\phi}})$ using Notation \ref{not:phi,gamma} and $y^{\vec{\phi}} = ((-1)^{\phi_1} y_1, \ldots, (-1)^{\phi_m} y_m)$ by
\begin{align*}
 \tau^{F_1,F_2}_{y}f(x)&=\frac1{4^m}\sum_{\vec{j} \in J}\sum_{\vec{\phi},\vec{\gamma} \in\{0,1\}^m}
c_{\vec{j},\vec{\phi},\vec{\gamma}} \\
  & \quad \times
\prod_{k=\mu}^1(i_k)^{j_{1,k}}\prod_{k=1}^\mu (-i_k)^{j_{2,k}}\prod_{k=1}^\mu(-i_k)^{j_{3,k}} \\
& \quad\times \tau_{y^{\vec{\phi}}}f^{\vec{\gamma}}(x) \\
  & \quad \times \prod_{k=\mu+1}^m(-i_k)^{j_{3,k}}\prod_{k=\mu+1}^m(-i_k)^{j_{2,k}}\prod_{k=m}^{\mu+1}(i_k)^{j_{1,k}}.
\end{align*}
\end{theorem}

\begin{proof}
Similar to the proof of Theorem \ref{mustardconvGFT}.
\end{proof}

Using this result, we can now give an explicit expression for the convolution product defined using the translation operator. 

\begin{corollary}
\label{transconvgft}
Let $J=\{0,1\}^{4\times m}$ with $j_{1,k}+j_{2,k}+j_{3,k}\in\{0,2\}$ and $j_{4,k}=0$ for all $k=1,\ldots,m$ be a set of multi-indices. The convolution $*^\tau_{F_1,F_2}$ defined by
\[
(f*^\tau_{F_1,F_2}g)(x):=\int_{\mR^m}{f(y)} \left[ \tau^{F_1,F_2}_{y} g(x) \right] dy
\]
with $\tau^{F_1,F_2}_{y}$ from Definition \ref{def:transl} can be expressed as a sum of classical convolutions using Notation \ref{not:phi,gamma} by
\begin{align*}
& (f*^\tau_{F_1,F_2}g)(x)\\
=&\frac1{4^m}\sum_{\vec{j}\in J}\sum_{\vec{\phi},\vec{\gamma} \in\{0,1\}^m}
c_{\vec{j},\vec{\phi},\vec{\gamma}} \\
& \quad \left( f^{\vec{\phi}}  \prod_{k=\mu}^1(i_k)^{j_{1,k}}\prod_{k=1}^\mu (-i_k)^{j_{2,k}}\prod_{k=1}^\mu(-i_k)^{j_{3,k}} \right)\\
& \quad \ast  \left( g^{\vec{\gamma}}  \prod_{k=\mu+1}^m(-i_k)^{j_{3,k}}\prod_{k=\mu+1}^m(-i_k)^{j_{2,k}}\prod_{k=m}^{\mu+1}(i_k)^{j_{1,k}} \right)(x).
\end{align*}
\end{corollary}

Note that a similar expression as in Corollary \ref{transconvgft} can be obtained for
\[
\int_{\mR^m}{\left[ \tau^{F_1,F_2}_{y} f(x) \right]} \  g(y) \  dy.
\]


\subsection{Special case: the two-sided quaternionic Fourier transform}
\label{secQFT}

\subsubsection{Definition of the qFT}

The quaternion algebra $\mH$ is isomorphic with the Clifford algebra $\cC l_{0,2}$ under the identification ${\bf i} = e_{1}$, ${\bf j} =e_{2}$ and ${\bf k} = e_{1} e_{2}$.

Let $\mu, \nu \in \mH$ be quaternions with $\mu^2 = \nu^2 =-1$. Then, following \cite{HS}, we define the two-sided qFT as
\begin{align*}
&\cF^{\mu,\nu}(f) (y_1,y_2) \\
&:= (2 \pi)^{-1} \int_{\mR^2} e^{-\mu x_1 y_1} f(x_1,x_2) e^{-\nu x_2 y_2} dx_1 dx_2
\end{align*}
for functions $f \in L_1(\mR^2; \mH)$
where we have introduced a different normalization $(2 \pi)^{-1}$. The first definition of this two-sided transform, with $\mu={\bf j}$ and $\nu={\bf k}$, was introduced in the Ph.D. thesis \cite{EPHD}, see also \cite{E1}. In earlier work, a one sided version was given by Ernst et al. \cite{EBW} and by Delsuc \cite{Delsuc}, although these authors use an adaptation of the quaternion algebra. 
The applicability of the qFT to color image processing was first demonstrated in \cite{S} using a discrete version. 
 At that point, the switch was made to two general orthogonal axes $\mu$ and $\nu$ instead of ${\bf j}$ and ${\bf k}$. Indeed, for color image processing there is an arbitrary but preferred axis of the grey-line in the color space, so the transform kernel axes are generally aligned to or perpendicular to this axis. At the same time (\cite{SE}), a change was again made to one-sided transforms, mostly driven by the complexity of the resulting operational formula when using the two-sided qFT definition. Finally, the orthogonality condition on $\mu$ and $\nu$ was relaxed in \cite{HS}. For a recent review on the use of the qFT in image processing, we refer the reader to \cite{ES}.

 \subsubsection{Convolution for the qFT}

 In this section, we discuss the Mustard and generalized translation definitions of the convolution product for the qFT. First, we give the interaction of the qFT with the ordinary convolution. To that aim, we need the following definition.

\begin{definition}\label{d:c}
For an invertible multivector $b\in\cC l_{0,m}$ and an arbitrary multivector $a\in\cC l_{0,m}$ we define the commutative and anticommutative part of $a$ with respect to $b$ by
\begin{align*}
 a_{c^0(b)}=&\frac12(a+b^{-1}ab),\\
a_{c^1(b)}=&\frac12(a-b^{-1}ab).
\end{align*}
\end{definition}
Using this definition, the quaternionic Fourier transform of the ordinary convolution $(f*g)(x)$ is given by
\begin{align*}
&\cF^{\mu,\nu}(f*g)(u)\\
&= 2 \pi  \sum_{j,k=0}^1(\cF^{\mu,(-1)^k\nu}(f)(u))_{c^{j}(\mu)}
\cF^{(-1)^j\mu,\nu}(g_{c^{k}(\nu)})(u).
\end{align*}
This formula is a special case of the convolution theorem for a general GFT, obtained in \cite{BU5}. In the discrete case, a similar formula was earlier obtained in \cite{ES}. Note that, in particular,
\begin{equation}
\label{convQ}
\cF^{\mu,\nu} \left(f * g \right) \neq 2 \pi \cF^{\mu,\nu}(f) \cF^{\mu,\nu}(g).
\end{equation}

The Mustard convolution, as obtained explicitly in Theorem \ref{mustardconvGFT} for an arbitrary GFT, reduces in the case of the two-sided qFT to
\begin{align*}
 (f \ast_{q} g)(x) &= \frac{1}{16} \sum_{\stackrel{\vec{j}\in\{0,1\}^{4\times 2},}{\forall k=1,2:j_{1,k}+j_{2,k}+j_{3,k}\in\{0,2\},j_{4,k}=0}} \\
& \quad \sum_{\vec{\phi}, \vec{\gamma} \in\{0,1\}^2}  c_{\vec{j},\vec{\phi},\vec{\gamma}}  \left(  (\mu)^{j_{1,1}}    (-\mu)^{j_{2,1}}  f^{\vec{\phi}}  (-\nu)^{j_{2,2}} \right) \\
& \quad \ast  \left( (-\mu)^{j_{3,1}}  g^{\vec{\gamma}}   (- \nu)^{j_{3,2}} (\nu)^{j_{1,2}} \right) (x).
\end{align*}
It can be checked that this expression coincides with the more symmetric formula
\begin{align*}
 (f \ast_{q} g)(x)& = \frac{1}{4} \sum_{j_1,j_2=0}^1 \sum_{k_1,k_2=0}^1 c_{j_1,j_2,k_1,k_2} \\
  & \quad \times \bigl( (\mu^{j_1} f^{k_1} \nu^{j_2}) \ast (\mu^{j_1} g^{k_2} \nu^{j_2}) \bigr)(x)
\end{align*}
where 
\[
c_{j_{1}, j_{2}, k_{1},k_{2}}:= (-1)^{(k_2+1) \delta_{j_1,1}}  \ (-1)^{(k_1+1) \delta_{j_2,1}}
\]
and
\begin{align*}
f^{k_1}(x_1,x_2) &= f(x_1,(-1)^{k_1} x_2), \qquad k_1 \in \{0,1\}\\ 
g^{k_2}(x_1,x_2) &= g((-1)^{k_2}x_1,x_2), \qquad k_2\in \{0,1\}.
\end{align*}
The advantage of the Mustard definition is that in this case, contrary to formula (\ref{convQ}), the following holds:
\[
\cF^{\mu,\nu} \left(f *_{q} g \right) = 2 \pi \cF^{\mu,\nu}(f) \cF^{\mu,\nu}(g).
\]

The situation is quite different for the convolution defined using the generalized translation. First of all, it can easily be proven (see e.g. \cite{DBQ}) for any function $f$ that the generalized translation in the case of the qFT coincides with geometric translation:
\[
  \tau^{\mu,\nu}_{y}f(x)=  \tau_{y}f(x).
\]
As a consequence, the associated convolution also coincides with the classical convolution, i.e.
\[
(f*^\tau_{ \mu ,  \nu }g)(x) = (f \ast g)(x).
\]
This can also be proven starting from the result in Corollary \ref{transconvgft}.

\section{Conclusions}

In this paper, we have studied two conceptual ways of defining convolution products, namely using the method of Mustard and using the generalized translation operator. We applied these ideas to two important families of hypercomplex Fourier transforms. A summary of our results can be found in Table \ref{tableconv}.

We expect that in particular the Mustard convolution will find many applications in color image processing. Currently, the design of filters for such images using hypercomplex methods is hindered by the lack of a proper convolution theorem. Our results now enable the development of a complete  theory of linear system design in the quaternionic case (as well as in higher dimensions, which may be interesting for other applications). A Fourier domain analysis, using the Mustard convolution, of the color edge filter constructed in \cite{S_CE} will serve as the guiding example to achieve this.

\section*{Acknowledgements}
The authors would like to thank Todd Ell and Stephen Sangwine for email communication about the qFT, as well as Eckhard Hitzer for communication about the GFT and proofreading Section 4.

\newpage 

\vspace{2mm}~

\newpage

\begin{table}[h]
\caption{Results on convolutions for the two approaches \textbf{A} and \textbf{B}}
\begin{center}
\begin{tabular}{c||c|c|}
& Approach \textbf{A}& Approach \textbf{B}\\
& eigenfunction approach & generalized roots of $-1$ approach\\ 
&&\\
\hline
&&\\
Definition& $\cF_K$&$ \cF_{F_1,F_2}$\\
& formulas (\ref{genkernel}), (\ref{gentransform})&Definition \ref{def:gft}\\
Eigenfunctions&$\psi_{j,k,\ell}$& $\psi_{j_1,j_2, \ldots, j_m}$\\
Eigenvalues & complex numbers & elements from $\cC l_{0,m}$\\
& Theorem \ref{eigenvalues}& Theorem \ref{eigenvalsGFT}\\
&&\\
\hline
&&\\
Mustard convolution&Definition \ref{conv2}& Definition \ref{def:conv_must}\\
Expression for convolution& Proposition \ref{explicitconv2}& Theorem \ref{mustardconvGFT}\\
&&\\
\hline
&&\\
Generalized translation&$\tau^{K}_{y}$&$\tau^{F_1,F_2}_{y}$\\
 & Definition \ref{defTrans} & Definition \ref{def:transl} \\
action on: & radial functions & arbitrary functions\\
& Theorem \ref{transradial} & Theorem \ref{transGFT}\\
Translation convolution&Definition \ref{conv1} &Corollary \ref{transconvgft}\\
Expression for convolution & Proposition \ref{explicitconv1} & Corollary \ref{transconvgft}
\end{tabular}
\end{center}
\label{tableconv}
\end{table}

\end{document}